\documentclass[10pt,a4paper,oneside]{amsart}

\usepackage[a4paper,inner=2.5cm,outer=2cm,top=3cm,bottom=3cm,pdftex]{geometry}

\usepackage{color}
\usepackage{amsmath,amssymb,amsfonts}
\usepackage{graphicx}
\usepackage{enumerate}

\usepackage[utf8]{inputenc}
\usepackage[T1]{fontenc}
\usepackage[colorlinks=true,linkcolor=black,citecolor=black]{hyperref}

\newtheorem{thm}{Theorem}[section]
\newtheorem{cor}[thm]{Corollary}
\newtheorem{lem}[thm]{Lemma}
\newtheorem{remark}[thm]{Remark}
\newtheorem{prop}[thm]{Proposition}
\newtheorem{defn}[thm]{Definition}

\theoremstyle{plain}
\newtheorem{rem}[thm]{Remark}

\numberwithin{equation}{section}

\newcommand{\eps}{\varepsilon}

\newcommand{\RR}{\mathbb{R}}
\newcommand{\TT}{\mathbb{T}}

\newcommand{\dd}{\partial}

\newcommand{\Pcal}{\mathcal{P}}

\newcommand{\Ucal}{\mathcal{U}}

\newcommand{\divv}{\mbox{div }}

\newcommand{\abs}[1]{\left\vert#1\right\vert}
\newcommand{\set}[1]{\left\{#1\right\}}

\newcommand{\pint}[1]{\left[#1\right]}
\newcommand{\pare}[1]{\left(#1\right)}
\newcommand{\norm}[1]{\big\Vert#1\big\Vert}

\newcommand{\inner}[1]{\left(#1\right)}
\newcommand{\comii}[1]{\left<#1\right>}
\newcommand{\com}[1]{\left[#1\right]}
\newcommand{\reff}[1]{(\ref{#1})}

\usepackage{empheq}

\definecolor{myyellow}{RGB}{255,255,190}
\definecolor{myrose}{RGB}{255,218,185}
\definecolor{mygreen}{RGB}{152,251,152}

\newlength\mytemplen
\newsavebox\mytempbox

\makeatletter

\newcommand\yellowbox{%
    \@ifnextchar[
       {\@yellowbox}%
       {\@yellowbox[0pt]}}

\def\@yellowbox[#1]{%
    \@ifnextchar[
       {\@@yellowbox[#1]}%
       {\@@yellowbox[#1][0pt]}}

\def\@@yellowbox[#1][#2]#3{
    \sbox\mytempbox{#3}%
    \mytemplen\ht\mytempbox
    \advance\mytemplen #1\relax
    \ht\mytempbox\mytemplen
    \mytemplen\dp\mytempbox
    \advance\mytemplen #2\relax
    \dp\mytempbox\mytemplen
    \colorbox{myyellow}{\hspace{1em}\usebox{\mytempbox}\hspace{1em}}}

\newcommand\greenbox{%
    \@ifnextchar[
       {\@greenbox}%
       {\@greenbox[0pt]}}

\def\@greenbox[#1]{%
    \@ifnextchar[
       {\@@greenbox[#1]}%
       {\@@greenbox[#1][0pt]}}

\def\@@greenbox[#1][#2]#3{
    \sbox\mytempbox{#3}%
    \mytemplen\ht\mytempbox
    \advance\mytemplen #1\relax
    \ht\mytempbox\mytemplen
    \mytemplen\dp\mytempbox
    \advance\mytemplen #2\relax
    \dp\mytempbox\mytemplen
    \colorbox{mygreen}{\hspace{1em}\usebox{\mytempbox}\hspace{1em}}}

\newcommand\rosebox{%
    \@ifnextchar[
       {\@rosebox}%
       {\@rosebox[0pt]}}

\def\@rosebox[#1]{%
    \@ifnextchar[
       {\@@rosebox[#1]}%
       {\@@rosebox[#1][0pt]}}

\def\@@rosebox[#1][#2]#3{
    \sbox\mytempbox{#3}%
    \mytemplen\ht\mytempbox
    \advance\mytemplen #1\relax
    \ht\mytempbox\mytemplen
    \mytemplen\dp\mytempbox
    \advance\mytemplen #2\relax
    \dp\mytempbox\mytemplen
    \colorbox{myrose}{\hspace{1em}\usebox{\mytempbox}\hspace{1em}}}

\makeatother

\begin{document}

\title{Boundary layer analysis  for the fast horizontal rotating fluids}

\author[W. Li, V.-S. Ngo \& C.-J. Xu]
{Wei-Xi Li, Van-Sang Ngo and Chao-Jiang Xu}
\date{\today}
\address[W.-X. Li]{School of Mathematics and Statistics,  and Computational Science Hubei Key Laboratory, Wuhan University,
Wuhan 430072, China  }
\email{wei-xi.li@whu.edu.cn}
\address[V.-S. Ngo]{Universit\'e de Rouen, CNRS UMR 6085, Laboratoire de Math\'ematiques, 76801 Saint-Etienne du Rouvray, France} 
\email{van-sang.ngo@univ-rouen.fr}

\address[ C.-J. Xu]{ 
 Universit\'e de Rouen, CNRS UMR 6085, Laboratoire de Math\'ematiques, 76801 Saint-Etienne du Rouvray, France\\
and\\
 School of Mathematics and Statistics,   Wuhan University,
Wuhan 430072, China } 
\email{Chao-Jiang.Xu@univ-rouen.fr}


\begin{abstract}

It is well known that, for fast rotating fluids with the axis of rotation being  perpendicular  to the boundary, the boundary layer is of Ekman-type, described by a linear ODE system. In this paper we consider fast rotating fluids, with the axis of rotation being parallel to the boundary. We show that the corresponding boundary layer is describe by a nonlinear, degenerated PDE system which is similar to the  $2$ -D  Prandtl system. Finally, we prove the well-posedness  of the governing system of the boundary layer in the space of analytic functions with respect to tangential variable.
    \end{abstract}

\maketitle

\section{Introduction}\label{section1}

The incompressible Navier-Stokes equation coupled with a large Coriolis term reads
\begin{eqnarray*}
\left\{
    \begin{aligned}
        &\partial_tu^{\varepsilon} - \nu\Delta u^{\varepsilon} + u^{\varepsilon}\cdot\nabla u^{\varepsilon} + \frac{\omega\times u^{\varepsilon}}{\varepsilon} + \nabla p^{\varepsilon} = 0,&&\\
        &\divv u^{\varepsilon} = 0, &&\\
        &u^{\varepsilon}|_{t=0}= u^\varepsilon_0, &&
    \end{aligned}
    \right.	
\end{eqnarray*}
with Dirichlet boundary condition, where $ \frac{\omega\times u_\eps}{\eps}$ stands for the Coriolis force and $\omega$ is the rotation vector, $\eps^{-1}$ the rescaled speed of rotation, $\nu$ the viscosity  coefficients.   The above system   is sufficient to describe the rotation fluids which is a significant part of geophysics.  Due to the earth's self-rotation, we can't neglect the Coriolis force in order to model the oceanography and meteorology dealing with large-scale magnitude.  When   the fluid  is between  a strip  and  the direction of rotation is not parallel to the boundary, we have the well-developed   Ekman layers  to match the interior flow
with Dirichlet boundary condition, cf. \cite{CDGG3,CDGGbook, grenier-masmoudi, masmoudi-2} and the references therein.    The situation will be more complicated when the direction of rotation is parallel to the boundary, considering cylinder  for instance and  letting  the fluid rotate  around the vertical axis. Then we will have   two types of boundaries,  the  horizontal boundary layer which is  Ekman layers and the vertical  boundary layers for  which much less is known, despite various studies \cite{CDGGbook, ste,vooren}. We refer to \cite{CDGGbook} for detailed discussions on the problem of vertical boundary layers.    

In this paper, we consider the fast rotating viscous fluids  where the the axe of rotation is horizontal  with respect to the boundary.  We prove that the governing equation  for boundary layer is nonlinear PDE system which is similar to classical  $2$ -D  Prandtl boundary layer system,  and we also obtain  the well-posedness  of this vertical  boundary layers in the space of analytic functions.

  As a  preliminary step we first consider the half space case  $\mathbb R^3_+ = \mathbb{R}^2_h \times \mathbb{R}_+$.  More precisely, we consider the following system
\begin{displaymath}
	\label{NSCeps}\tag{N-S$_\varepsilon$}
	\left\{
    \begin{aligned}
        &\partial_tu^{\varepsilon} - \nu\Delta u^{\varepsilon} + u^{\varepsilon}\cdot\nabla u^{\varepsilon} + \frac{e_2\times u^{\varepsilon}}{\varepsilon} + \nabla p^{\varepsilon} = 0 &&\mbox{in }\;  \mathbb{R}^2_h \times \mathbb{R}_+, \;\forall t \geq 0\\
        &\divv u^{\varepsilon} = 0 &&\mbox{in }\; \mathbb{R}^2_h \times \mathbb{R}_+, \;\forall t \geq 0\\
        &u^\varepsilon|_{x_3=0} = 0 &&\mbox{on }\; \mathbb{R}^2_h\\
        &u^{\varepsilon}|_{t=0}= u^\varepsilon_0, &&\mbox{in }\;  \mathbb{R}^2_h \times \mathbb{R}_+\, .
    \end{aligned}
    \right.
\end{displaymath}
where $e_2 = (0,1,0)$ is the unit horizontal vector, $\nu>0$ the coefficient of viscosity of fluids and $\varepsilon$ the Rossby number. These equations describe the evolution of an incompressible three-dimensional viscous fluid in a rotating frame, $\frac{e_2\times u^{\varepsilon}}{\varepsilon}$ being the Coriolis force due to the rotation at high frequency $\varepsilon^{-1}$. According to the Taylor-Proudman theorem \cite{taylor}, the fast rotation penalize the movement of the fluid in the direction of the rotation axis. As a consequence, the fluid has tendency to move in columns, parallel to the rotation axis, which are widely known as the Taylor columns. This phenomenon is   well-known in oceanography and meteorology, which is observed in many large-scale atmospheric and oceanic flows. In mathematical point of view, when $\varepsilon$ goes to zero, the rotation term $\frac{e_2\times u^{\varepsilon}}{\varepsilon}$ becomes large and can only be balanced by the pressure. This means that, if $u$ is the (formal) limit of $u^\varepsilon$, as $\varepsilon \rightarrow 0$, then $e_2 \times u$ need to be a gradient term, which implies that $u$ is independent of $x_2$ (more explanations will be found in Section \ref{section2}). In this paper, we will consider the case where the initial data are well prepared, \emph{i.e.} $u^\varepsilon_0$ do not depend on $x_2$.

   When there is no Coriolis force, the zero-viscosity limit for the Navier-Stokes equations for incompressible fluids in a domain with boundary, with non-slip boundary conditions, is a challenging problem due to the formation of a boundary layer which is governed by the Prandtl equations (\cite{oleinik}). The mathematical analysis theory of Prandtl equation is also a challenging problem, see \cite{awxy, e-1,e-2,GV-D,M-W} and references therein.  Far from the boundary,  the inviscid  limit problem was treated by several authors; we can refer, for instance, to Swann \cite{swann} and Kato \cite{kato-2}. In another work, Kato \cite{kato-1} gives some equivalent formulations of this problem in the case of bounded domains, showing that the convergence to the Euler system is equivalent to the fact that the $L^2$ strength of the boundary layer goes to $0$. Caflisch \& Sammartino \cite{samm} solved the problem for analytic solutions on a half space by solving the Prandtl equations via abstract Cauchy-Kowaleskaya theorem. We also refer to \cite{GMM, GN, maekawa}  and the references therein for the recent progress  on the inviscid limit of the Navier-Stokes equations when the initial vorticity is located away from the boundary.   On the other hand,  another commonly used boundary conditions are  Navier-type slip boundary conditions,  in which case
   the vanishing viscosity limit is rigorously justified; cf.   \cite{Lions, wangwangxin, XiaoXin, Yu} and references therein.

We want to say a few words to compare the system \eqref{NSCeps} with the case where the rotation axis is vertical with respect to the boundary (the rotation axis is in the direction of $e_3 = (0,0,1)$ instead of $e_2$). If the domain considered is between two parallel plates ($\TT^2\times [0,1]$ or $\RR^2 \times [0,1]$), it was proved in Grenier \& Masmoudi \cite{grenier-masmoudi}, Masmoudi \cite{masmoudi-2, masmoudi-3} and Chemin \emph{et al.} \cite{CDGG3} that for the rotating fluids with anisotropic  viscosity $-\nu \Delta_h-\eps \partial_{x_3}^2$,  all the weak solutions of Navier-Stokes equation converge to the solution of the  2D Euler or 2D Navier-Stokes  system (with damping term - effect of the Ekman pumping). The vertical rotation and the specific form of the domain (between two parallel plates) permit to explicitly construct the boundary layer velocity term from the interior velocity term (which satisfies a 2D damped Euler system), without using the Prandtl equations.  We also want to mention the work of Dalibard and G\'erard-Varet 
\cite{DaGV} in the case of fast rotating fluids on a rough domain with 
non-slip boundary conditions. The boundary layer is also proved to be of size $\varepsilon$ (contrary to the case of Prandtl equations where the boundary layer is of size $\sqrt{\varepsilon}$). We also refer to a series of work for the rotating fluids with anisotopic viscosity (see for exemple \cite{CDGG}, \cite{CDGG2}, \cite{IG1}, \cite{IGLSR}, \cite{gongguowang},\cite{Iftimie3}, \cite{VSN}, \cite{Paicu1}).

We want to emphasize that the formation of the boundary layers in the case of vertical rotation axis is due to the incompatibility of the Dirichlet boundary conditions with the columnar movement of the limit fluid (as $\varepsilon \rightarrow 0$). Indeed, as the rotation axis is $e_3$, the limiting velocity of the fluid is independent of $x_3$, and so, the Dirichlet boundary conditions imply that the limit velocity should be zero. This incompatibility leads to the fact that a thin layer (Ekman's layer) is formed near the boundary, and the fluid's evolution is violent in this small scale zone, in a way that stops the fluid on the boundary.

In the case of horizontal rotation axis (in the direction of $e_2$), the incompatibility of boundary conditions  will be more complicated, because of the fact that the limit velocity is independent of $x_2$ instead of $x_3$. In Section \ref{section2}, we prove that the limit system is a 2D Euler-like system. This means that we are no longer in the case considered by Ekman. The techniques of \cite{grenier-masmoudi} and \cite{CDGG3} do not work and we can not explicitly calculate the boundary layer. The fast rotation only penalizes the fluid motion in the $x_2$ direction, and leads to a problem very close to the inviscid limit of two-dimensional Navier-Stokes system. It is then relevant to look for a boundary layer of size $\sqrt{\varepsilon}$ and we will show in Section \ref{section2} that in this boundary layer of size $\sqrt{\varepsilon}$, the fluid velocity actually satisfies a two-dimensional Prandtl-like system. Finally, we remark that in this paper, we only consider the case where $\nu=\varepsilon$. Indeed, as explained in \cite{grenier-masmoudi} and also in \cite{CDGGbook}, if the ratio $\nu/\varepsilon$ goes to infinity, the fluid rapidly stops after a few evolutions. It is then more interesting to consider the case where $\nu \lesssim \varepsilon$, which moreover better fits physical observations. 

In this work, we study the formation of the boundary layer when $\nu=\varepsilon \to 0$. We suppose the existence of a boundary layer of size $\sqrt{\varepsilon}$ near the boundary $\set{x_3=0}$ of $\RR^3_+$.  
We will derive the limit equation and the boundary layer equation by using a formal asymptotic expansion in the Section \ref{section2}. 
We refer to the book of Pedlovsky \cite{pedlovsky} for more detail about this formal expansion.
To this end, we suppose that the solution of \eqref{NSCeps} accepts the following asymtotic expansion
\begin{empheq}{align}
	\label{eq:AnsatzU}
	u^\varepsilon(t,x_1,x_2,x_3) &= \sum_{j=0}^1 \varepsilon^{\frac{j}2} \pint{u^{I,j}(t,x_1,x_2,x_3) + u^{B,j}\pare{t,x_1,x_2,\frac{x_3}{\sqrt{\varepsilon}}}} + \cdots,\\
	\label{eq:AnsatzP}
	p^\varepsilon(t,x_1,x_2,x_3) &= \sum_{j=-2}^1 \varepsilon^{\frac{j}2} p^{I,j}(t,x_1,x_2,x_3) + \sum_{j=-2}^0 \varepsilon^{\frac{j}2} p^{B,j}\pare{t,x_1,x_2,\frac{x_3}{\sqrt{\varepsilon}}}  +\cdots,
\end{empheq}
where $u^{B,j}(t,x_1,x_2,y)$ and $p^{B,j}(t,x_1,x_2,y)$ exponentially go to zero as $y\stackrel{\rm def}{=}\frac{x_3}{\sqrt{\varepsilon}} \to +\infty$. The remaining terms is supposed to be very small (at least of order 3).

Throughout this paper, we will always use $\dd_t$, $\dd_i$ (or $\dd_{x_i}$), $i=1,2,3$, and $\dd_y$ to respectively denote the derivatives with respect to the time variable $t$, the space variables $x_i$, $i=1,2,3$, and the boundary layer variable $y = \frac{x_3}{\sqrt{\varepsilon}}$. Using the above asymptotic expansion, we first deduce that  the behavior of the fluid near the boundary  is governed by the following 2D Prandtl-like equation
  \begin{empheq}{equation}
    \label{eq:PrRotE2} \tag{P1}
    \left\{
    \begin{aligned}
    	&\partial_t \Ucal^{p,0}_1 - \partial_y^2 \Ucal^{p,0}_1 + \Ucal^{p,0}_1 \partial_1 \Ucal^{p,0}_1 + \Ucal^{p,1}_3 \partial_y \Ucal^{p,0}_1 + \partial_1 p^{B,0} + \overline{\dd_1 p^{I,0}} = 0,\\
    	&\partial_t \Ucal^{p,1}_3 - \partial_y^2 \Ucal^{p,1}_3 + \Ucal^{p,0}_1 \partial_1 \Ucal^{p,1}_3 + \Ucal^{p,1}_3 \partial_y \Ucal^{p,1}_3 + \overline{\partial_3 p^{I,1}} + y\overline{\partial_3^2 p^{I,0}} = 0,\\
    	&\partial_1 \Ucal^{p,0}_1 + \partial_y \Ucal^{p,1}_3 = 0,\\
    	&\Ucal^{p, 0}_1|_{y=0} =0,\quad \lim_{y\to +\infty} \Ucal^{p,0}_1(t,x_1,y)=\overline{u^{I,0}_1},\\
	&\Ucal^{p, 1}_3|_{y=0} =0,\quad \partial_y \Ucal^{p,1}_3 |_{y=0}= 0,\\    	
	& ( \Ucal^{p,0}_1,   \Ucal^{p,1}_3)|_{t=0} =  ( \Ucal^{p,0}_{1, 0},   \Ucal^{p,1}_{3, 0}),
    \end{aligned}
	\right.
\end{empheq}
with the unknown functions $\inner{\Ucal^{p,0}_1, \Ucal^{p,1}_3, p^{B,0}},$  and the horizontal second component satisfies   a  parabolic type equation
\begin{empheq}{equation}
    \label{eq:PrRotE2-b} \tag{\ref{eq:P2}}
    \left\{
    \begin{aligned}
    	&\partial_t \Ucal^{p,0}_2 - \partial_y^2 \Ucal^{p,0}_2 + \Ucal^{p,0}_1 \partial_1 \Ucal^{p,0}_2 + \Ucal^{p,1}_3 \partial_y \Ucal^{p,0}_2 = 0\\
    	&\Ucal^{p, 0}_2|_{y=0} =0,\quad \lim_{y\to +\infty} \Ucal^{p,0}_2(t,x_1,y)\big) =\overline{u^{I,0}_2}\\
    	& \Ucal^{p, 0}_2 |_{t=0} =  \Ucal^{p, 0}_{2, 0}.
    \end{aligned}
	\right.
\end{empheq}
Here
$$
\partial_2 \Ucal^{p,0}_1 = \partial_2 \Ucal^{p,0}_2 = \partial_2 \Ucal^{p,1}_3 = 0.
$$
Here, we emphasize the ``Prandtl-like'' property of our system by using the new unknown functions 
\begin{align*}
	&\Ucal^{p,0}_j= u^{B,0}_j + \overline{u^{I,0}_j}, \qquad j=1, 2\\
	&\Ucal^{p,1}_3 = u^{B,1}_3 + \overline{u^{I,1}_3} + y \overline{\partial_3 u^{I,0}_3}
\end{align*}
where $\overline{u^{I,j}},  \overline{p^{I,j}}, j=1, 2 $ are the values on the boundary of the tangential velocity and
pressure of the outflow satisfying the Bernoulli-type law
$$
    \left\{
    \begin{aligned}
		&\partial_t \overline{u^{I,0}_1}+\overline{u^{I,0}_1}\partial_1\overline{u^{I,0}_1}+ \overline{\partial_1p^{I,0}}=0\\
		&\partial_t \overline{u^{I,0}_2}+\overline{u^{I,0}_1}\partial_1\overline{u^{I,0}_2}+ \overline{\partial_2p^{I,0}}=0\\
		&\partial_t \overline{u^{I,1}_3}+\overline{u^{I,0}_1}\partial_1\overline{u^{I,1}_3}+\overline{u^{I,1}_3}\partial_3\overline{u^{I,0}_3}+ \overline{\partial_3p^{I,1}}=0
  \end{aligned}
	\right.
$$
which is the restriction of the Euler system and linearized Euler system on the boundary $x_3=0$, so that they depend only on the variavles $(t, x_1)$. More precise description will be found in Section \ref{section2}.

Note that the boundary layer equation  \eqref{eq:PrRotE2} look very close to that of classical 2D Prandtl equation, but the fast rotating produces the boundary layer pressure for the first components, so that the boundary layer equation \eqref{eq:PrRotE2} is now really a system of 3 equations with both the velocity $ ( \Ucal^{p,0}_1,   \Ucal^{p,1}_3)$ and the boundary pressure $p^{B, 0}$ to determined.  We remark that on one side, the first equation in  \eqref{eq:PrRotE2} admits the similar structure of Prandtl equation, i.e.,   the degeneracy in $x_1$ coupled with the nonlocal property  arising from the term $\Ucal^{p,1}_3 \partial_y \Ucal^{p,0}_1$,  so that the system \eqref{eq:PrRotE2} is quite similar  to Prandtl equation.  Therefore  we can only expect the local well-posedness  for  analytic  initial data if no additional assumptions are imposed.   On the other hand,  there is a crucial difference between  Prandtl equation and   the first equation in \eqref{eq:PrRotE2},   due to  the unknown pressure  $p^{B,0}.$   Recall the pressure term  in Prandtl equation  is from outflow and  can be defined  by the Bernoulli law,  so that the pressure therein is a given function and therefore  Prandtl equation is a kind of degenerate parabolic equation.  But here the situation is quite complicated since we have the unknown pressure $p^{B,0}$  in  \eqref{eq:PrRotE2},  which arises because of the fast rotation  parallel to the boundary,  and  can't be defined by   the Bernoulli law    anymore.  So the classical theory for Prandtl equation is not applicable directly to  our case  and moreover we can't follow the same strategy as in Prandtl equation to treat the the first equation in \eqref{eq:PrRotE2}.   To overcome the difficulty  due to the unknown pressure term in the first equation of \eqref{eq:PrRotE2}, we will  firstly solve the second equation for $\Ucal^{p,1}_3$, and then use the divergence-free property to find $\Ucal^{p,0}_1$ (see Section \ref{section3} for detail).  Finally we mention that the mathematical justification of the inviscid limit for solutions to \eqref{NSCeps}, is also complicated as classical Prandtl boundary  layer theory.     We only concentrate in this work on the well-posedness of boundary layer  and  will investigate this inviscid limit problem in the future work.

On the other hand, we will prove in Section \ref{section2} that the limiting velocity of the outer flow satisfies a   classical 2D Euler-type equation, which is,
\begin{equation}  
	\label{eq:Ordre0Ibis}
    \left\{
    \begin{aligned}
    	&\partial_t u^{I,0}_1 + u^{I,0}_1 \partial_1 u^{I,0}_1 + u^{I,0}_3 \partial_3 u^{I,0}_1 + \partial_1 p^{I,0} = 0\\
    	&\partial_t u^{I,0}_2 + u^{I,0}_1 \partial_1 u^{I,0}_2 + u^{I,0}_3 \partial_3 u^{I,0}_2 = 0\\
    	&\partial_t u^{I,0}_3 + u^{I,0}_1 \partial_1 u^{I,0}_3 + u^{I,0}_3 \partial_3 u^{I,0}_3 + \partial_3 p^{I,0} = 0\\
    	&\partial_2 u^{I,0}_1 = \partial_2 u^{I,0}_2 = \partial_2 u^{I,0}_3 = \dd_2 p^{I,0} = 0\\
    	&\partial_1 u^{I,0}_1 + \partial_3 u^{I,0}_3 = 0\\
    	&u^{I,0}_3|_{x_3=0} = 0\\
        &u^{I,0} |_{t=0}= u^{I,0}_0(x_1,x_3).
    \end{aligned}
    \right.
\end{equation}

In the system \reff{eq:Ordre0Ibis}, the components  $(u^{I, 0}_1, u^{I, 0}_3, p^{I, 0})$ satisfy exactly a 2-D incompressible Euler equation
on the half-plane, so that the existence and regularity in Gevery class of local in time solution is well know, (see Vicol \cite{KV} and references therein), but in the study of boundary layer equation, we need  some weighted on the tangential variables,  we cite in particular the results of \cite{cheng-li}.

\begin{defn}
Let $\frac{1}{2}<\ell\leq 1$ be given. 	We denote by  $\mathcal A_\tau$  the space of analytic functions with analytic radius  $\tau>0$, which is consist of      all functions $f\in  L^2(\RR^2_+)$ such that 
	\begin{equation*}
		\norm{f}_{\mathcal A_\tau}\stackrel{\rm def}{=} \sup_{\abs\alpha\geq 0} \frac{\tau^{\abs\alpha}}{\abs{\alpha}!}\norm{\comii {z}^\ell \partial_{z}^\alpha f}_{L^2(\mathbb R_+^2)} < +\infty.
	\end{equation*} 
\end{defn}

\begin{thm}[\cite{cheng-li}]
	\label{th:E2D}  Suppose that the initial data $u^{I,0}_0=(u^{I,0}_{1,0},u^{I,0}_{2,0}, u^{I,0}_{3,0})$ in \reff{eq:Ordre0Ibis} satisfies  
	\begin{eqnarray*}
		u^{I,0}_{1,0},\; u^{I,0}_{2,0},\; u^{I,0}_{3,0} \in \mathcal A_{\tau_0}	\end{eqnarray*}
	for some $\tau_0>0$,  the divergence-free condition and the compatibility condition.  Then  Euler-type system \reff{eq:Ordre0Ibis} admits a unique solution $(u^{I,0}_1, u^{I,0}_2, u^{I,0}_3)\in L^\infty\inner{[0,T];~\mathcal A_\tau}$ for some  $T>0$ and $\tau>0.$  
\end{thm}
\noindent The construction of the components $(u^{I, 0}_1, u^{I, 0}_3, p^{I, 0})$ is given in \cite{cheng-li}. The construction of $u^{I,0}_2$ is standard, using the classical theory of transport equation.

\bigskip

Now we list several estimates, which are just immediate consequences of the definition of $\norm{\cdot}_{\mathcal A_{\tau}}$ and Sobolev inequalities.   
For  $u^{I,0}_1\in L^\infty\inner{[0,T];~\mathcal A_\tau},$ we have, for all $p,q\geq 0,$
\begin{eqnarray}
	\label{eseuler}
	\norm{\comii{x_1}^\ell\partial_1^{p}\partial_3^q u^{I,0}_3(x_1,x_3)}_{L^\infty\inner{\mathbb R_+;~L^2(\mathbb R_{x_1})}}\leq C \norm{u^{I,0}_3}_{\mathcal A_\tau} \frac{(p+q+3)!}{  \tau^{p+q+3}}.
\end{eqnarray}
Using the equation 
\begin{eqnarray*}
	\partial_t u^{I,0}_1 + u^{I,0}_1 \partial_1 u^{I,0}_1 + u^{I,0}_3 \partial_3 u^{I,0}_1 + \partial_1 p^{I,0}=0, 
\end{eqnarray*}
we can calculate, by virtue of Leibniz formula, 
\begin{eqnarray}\label{dt}
&&\norm{\comii{x_1}^\ell \partial_t\partial_1^{p}\partial_3^q u^{I,0}_3}_{L^\infty\inner{\mathbb R_+;~L^2(\mathbb R_{x_1})}}\qquad\qquad\\
&&\leq C_\tau \inner{ \norm{u^{I,0}_1}_{\mathcal A_\tau}^2+\norm{u^{I,0}_1}_{\mathcal A_\tau}\norm{u^{I,0}_3}_{\mathcal A_\tau}+\norm{p^{I,0}}_{\mathcal A_\tau}} \frac{2^{p+q} (p+q)!}{\tau^{p+q}}\,.\nonumber
  	\end{eqnarray}

In order to completely give the solutions of the systems \eqref{eq:PrRotE2} and \eqref{eq:PrRotE2-b}, we also need the following linearized Euler system, which describes the evolution of the fluids in the interior part of the domain, far from the boundary, at the order $\sqrt{\varepsilon}$.

\begin{empheq}{equation} 
    \label{eq:Ordre0.5I}
    \left\{
    \begin{aligned}
		&\partial_t u^{I,1}_1 + u^{I,0}_1 \partial_1 u^{I,1}_1 + u^{I,0}_3 \partial_3 u^{I,1}_1 + u^{I,1}_1 \partial_1 u^{I,0}_1 + u^{I,1}_3 \partial_3 u^{I,0}_1 + \dd_1 p^{I,1} = 0\\
		&\partial_t u^{I,1}_2 + u^{I,0}_1 \partial_1 u^{I,1}_2 + u^{I,0}_3 \partial_3 u^{I,1}_2 + u^{I,1}_1 \partial_1 u^{I,0}_2 + u^{I,1}_3 \partial_3 u^{I,0}_2 = 0\\
		&\partial_t u^{I,1}_3 + u^{I,0}_1 \partial_1 u^{I,1}_3 + u^{I,0}_3 \partial_3 u^{I,1}_3 + u^{I,1}_1 \partial_1 u^{I,0}_3 + u^{I,1}_3 \partial_3 u^{I,0}_3 + \dd_3 p^{I,1} = 0\\
    	&\partial_2 u^{I,1}_1 = \partial_2 u^{I,1}_2 = \partial_2 u^{I,1}_3 = \dd_2 p^{I,1} = 0\\
    	&\partial_1 u^{I,1}_1 + \partial_3 u^{I,1}_3 = 0\\
        &u^{I,1}_3|_{x_3=0}= -u^{B,1}_3(t,x_1,0)\\
    	&u^{I,1}|_{t=0} = u^{I,1}_0(x_1,x_3).
    \end{aligned}
    \right.
\end{empheq}

For this linearized Euler system \reff{eq:Ordre0.5I}, we have 
 
\begin{thm}
	\label{th:E2Dlin} Let $\ell>1/2$, $\tau_0>0$ and $u^{B,1}_3(t,x_1,0)$ a given function such that
	$$\sum_{m\leq 2 } \norm{\comii{x_1}^\ell  \partial_1^m u^{B,1}_3(t,x_1,0)}_{L^2(\mathbb R_{x_1})}^2 + \sum_{ m\geq 3 } \com{\frac{\tau_0^{m-1}}{ (m-3)!}}^2\norm{\comii{x_1}^\ell  \partial_1^m u^{B,1}_3(t,x_1,0)}_{L^2(\mathbb R_{x_1})}^2 <+\infty.$$ 
	Suppose that the initial data $u^{I,1}_0=(u^{I,1}_{1,0},u^{I,1}_{2,0}, u^{I,1}_{3,0})$ in \reff{eq:Ordre0.5I} satisfies  the divergence-free condition, the compatibility condition and
	\begin{eqnarray*}
		u^{I,1}_{1,0},\;u^{I,1}_{2,0},\; u^{I,1}_{3,0}\in \mathcal A_{\tau_0}.
	\end{eqnarray*}
	Then the  linearized Euler system \reff{eq:Ordre0.5I} admits a unique solution $(u^{I,1}_1, u^{I,1}_2, u^{I,1}_3)\in L^\infty\inner{[0,T];~\mathcal A_\tau}$ for some  $T>0$ and $\tau>0.$ 
\end{thm} 
We remark that,   the compatibility condition ask
$$
u^{I,1}_{3,0} (x_1, 0)=-u^{B,1}_{3,0}(x_1, 0).
$$
It is exactly the non-slip condition 	of  \eqref{NSCeps} at order  $1$. 
Because of its linearity, treating the system \eqref{eq:Ordre0.5I} is still much easier than treating the system \eqref{eq:Ordre0Ibis}, even with the presence of the given boundary function $u^{B,1}_3(t,x_1,0)$. So, to prove Theorem \ref{th:E2Dlin}, we can simply follow the lines of the proof of Theorem \ref{th:E2D} as in \cite{cheng-li}.

Before giving the well-posedness results on \eqref{eq:PrRotE2} and \eqref{eq:PrRotE2-b}, 
  we need the following  weighted   analytic function spaces in tangential variable.    We also remark that there is no coupling between $(\Ucal^{p,0}_1,\Ucal^{p,1}_3)$ and $\Ucal^{p,0}_2$. Then, the strategy   consists in separately  solving the systems \eqref{eq:PrRotE2} and \eqref{eq:PrRotE2-b}.

  \begin{defn}\label{xrho} Let $1/2<\ell\leq 1$ be given throughout the paper.  
With each pair $(\rho, a) $ with $\rho>0$ and  $a>0$  we associate   
 a space $X_{\rho,a} $ of all functions $u(x_1,y)\in H^\infty(\mathbb R_{x_1};~H^2(\mathbb R_+))$ such that  
\begin{eqnarray*}
\sum_{m\leq 2 }  \inner{\sum_{0\leq j\leq 1}  \norm{\comii{x_1}^\ell e^{a y^2}\partial_1^m\partial_y^j u}_{L^2(\mathbb R_+^2)}^2}+\sum_{ m\geq 3 } \inner{\sum_{0\leq j\leq 1} \com{\frac{\rho^{m-1}}{ (m-3)!}}^2\norm{\comii{x_1}^\ell e^{a y^2}\partial_1^m\partial_y^j u}_{L^2(\mathbb R_+^2)}^2 }<+\infty,
\end{eqnarray*}
where we use the convention $0!=1$. 
  We endow  $X_{\rho,a}$ with the norm
\begin{eqnarray*}
 \abs{u}_{X_{\rho, a}}^2=\sum_{m\leq 2 }  \inner{\sum_{0\leq j\leq 1}  \norm{\comii{x_1}^\ell e^{a y^2}\partial_1^m\partial_y^j u}_{L^2(\mathbb R_+^2)}^2}+\sum_{ m\geq 3 } \inner{\sum_{0\leq j\leq 1} \com{\frac{\rho^{m-1}}{ (m-3)!}}^2\norm{\comii{x_1}^\ell e^{a y^2}\partial_1^m\partial_y^j u}_{L^2(\mathbb R_+^2)}^2 }.
\end{eqnarray*}
\end{defn}

The well-posedness of the system \eqref{eq:PrRotE2} 
can be stated as follows.

\begin{thm} 
\label{th:Prandtl}
Suppose that  the initial data
\begin{eqnarray*}
	\Ucal^{p,1}_{3,0} = u^{B,1}_{3,0} + \overline{u^{I,1}_{3,0}} + y \overline{\partial_3 u^{I,0}_{3,0}}
\end{eqnarray*}
 in  \reff{eq:PrRotE2} 
 satisfies that
\begin{eqnarray*}
u^{B,1}_{3,0}\in X_{\rho_0, a_0}, \quad u^{I,1}_{3,0}, \,\, u^{I,0}_{3,0}\in \mathcal A_{\tau_0}
\end{eqnarray*}
for some $a_0>0$, $\rho_0 >0$ and $\tau_0>0$ and 
$$
\Ucal^{p,0}_{1, 0}(x_1, y) =-\int_{-\infty}^{x_1} \partial_yu^{B,1}_{3, 0}(z, y)dz+ 
\overline{u^{I,0}_{1, 0}}(x_1).
$$  
Then there exist $T>0$, $\tau>0$   and a pair $\inner{\rho, a}$ with $\rho, a>0$, such that the system \eqref{eq:PrRotE2} admits a unique solution $(\Ucal^{p,0}_1,\Ucal^{p,1}_3, \partial_1p^{B, 0})$, and moreover  
\begin{eqnarray*}
	&&\Ucal^{p,1}_3 = u^{B,1}_3 + \overline{u^{I,1}_3} + y \overline{\partial_3 u^{I,0}_3}\\
	&&\Ucal^{p,0}_1(t,x_1,y) =-\int_{-\infty}^{x_1} \partial_yu^{B,1}_3(t,z,y)dz+ \overline{u^{I,0}_1}(t,x_1),
\end{eqnarray*}
with    $u^{B,1}_3\in L^\infty\inner{[0,T]; ~X_{\rho,a}}$ and  $u^{I,0}_1, u^{I,0}_3, u^{I,1}_3 \in L^\infty\inner{[0,T]; ~\mathcal A_\tau}.$
\end{thm}

\begin{remark}
\begin{enumerate}[(i)]
\item Here we consider the well prepared initial data, that is the initial data are independent of $x_2.$ 
\item 	We want to remark that once we find $\Ucal^{p,1}_3$, we can obtain $\Ucal^{p,0}_1$ using the divergence-free property in the third equation of the system \eqref{eq:PrRotE2}. \end{enumerate}
\end{remark}

 Let $\Ucal^{p,0}_1,\Ucal^{p,1}_3$ be the solutions to the system \eqref{eq:PrRotE2} given by the theorem above. Then we see    \eqref{eq:PrRotE2-b} is a linear parabolic equation,  and we have the following theorem concerned with its well-posedness.
 
\begin{thm}
	\label{th:SystemP2} Let $\rho_0 > 0$, $a_0 > 0$, $\tau_0 > 0$ be given. For any initial data 
	$$
	\Ucal^{p,0}_{2,0} = u^{B,0}_{2,0} + \overline{u^{I,0}_{2,0}}
	$$ 
	where $u^{B,0}_{2,0} \in X_{\rho_0, a_0}$ and $u^{I,0}_{2,0}\in \mathcal A_{\tau_0}$, there exist $T>0$, $0<\tau<\tau_0$ and $0<a<a_0$, such that the equation \eqref{eq:PrRotE2-b} admits a unique solution $\Ucal^{p,0}_2$ satisfying $\Ucal^{p,0}_2 = u^{B,0}_2 + \overline{u^{I,0}_2}$ with 
	$$
	u^{B,0}_2 \in L^{\infty}\pare{[0,T], X_{\rho_0, a}},\quad u^{I,0}_2 \in L^{\infty}\pare{[0,T], \mathcal A_\tau}.
	$$ 
\end{thm}

By the two theorems above  we obtain the well-posedness for the boundary layer equation of the system \eqref{NSCeps} in the frame of analytic space in tangential variable.

The paper is organized as follows. In section \ref{section2}, we formally derive the governing equations of the outer flow inside the domain and the systems \eqref{eq:PrRotE2} and \eqref{eq:PrRotE2-b} which describe the fluid motion inside the boundary layer. The sections  \ref{section3}-\ref{section4} are devoted to proving  the well-posedness of the system \eqref{eq:PrRotE2}. Finally, we give some brief ideas of the proof of Theorem \ref{th:SystemP2} for the  well-posedness of  equation  \eqref{eq:PrRotE2-b} in the section \ref{section5}.

 \section{Formal asymtotic expansion}\label{section2}

First of all, we want to give a few words to explain our special choice of the order of the expansions of the velocity and the pressure. Indeed, we remark that as for the formulation of Prandtl boundary layer equations, we are only interested in the leading orders which are necessary to allow us to formally obtain the governing equations of the evolution of the boundary layer. By using the asymptotic expansions \eqref{eq:AnsatzU} and \eqref{eq:AnsatzP}, we have the following asymptotic identities for the leading terms up to order $\varepsilon^{1/2}$ and all the remaining terms are of higher order in $\eps$.
\begin{empheq}{equation}
	\label{eqs:AnsatzSyst}
	\left\{
	\begin{aligned}
		\partial_t u^\varepsilon &\quad=\quad \sum_{j=0}^1 \varepsilon^{\frac{j}2} \pare{\partial_t u^{I,j} + \partial_t u^{B,j}} +\cdots\\
		-\varepsilon\Delta u^\varepsilon &\quad=\quad -\partial_y^2 u^{B,0} - \varepsilon^{\frac{1}2}\partial_y^2 u^{B,1} - \sum_{j=0}^1 \varepsilon^{1+\frac{j}2} \pare{\Delta u^{I,j} + \Delta_h u^{B,j}} +\cdots\\
		u^\varepsilon\cdot \nabla u^\varepsilon &\quad=\quad \sum_{j=0}^1 \varepsilon^{\frac{j-1}2} \pint{\sum_{k=0}^j \pare{u^{B,k}_3 + u^{I,k}_3} \partial_y u^{B,j-k}} + \sum_{j=0}^1 \varepsilon^{\frac{j}2} \pint{\sum_{k=0}^j \pare{u^{B,k}_h + u^{I,k}_h}\cdot \nabla_h u^{B,j-k}}\\
		&\qquad \qquad + \sum_{j=0}^1 \varepsilon^{\frac{j}2} \pint{\sum_{k=0}^j \pare{u^{B,k} + u^{I,k}}\cdot \nabla u^{I,j-k}} + \cdots\\
		\frac{e_2\times u^\varepsilon}{\varepsilon} &\quad=\quad \sum_{j=0}^1 \varepsilon^{\frac{j}2-1} \pint{\begin{pmatrix} u^{B,j}_3\\0\\-u^{B,j}_1 \end{pmatrix} + \begin{pmatrix} u^{I,j}_3\\0\\-u^{I,j}_1 \end{pmatrix}} + \cdots\\
		\nabla p^\varepsilon &\quad=\quad \varepsilon^{-\frac{3}2} \begin{pmatrix} 0\\0\\ \partial_y p^{B,-1} \end{pmatrix} + \sum_{j=-2}^{-1} \varepsilon^j \begin{pmatrix} \partial_1 p^{B,j}\\ \partial_2 p^{B,j}\\ \partial_y p^{B,j+1} \end{pmatrix} + \begin{pmatrix} \partial_1 p^{B,0}\\ \partial_2 p^{B,0}\\ 0 \end{pmatrix} + \sum_{j=-2}^1 \varepsilon^{\frac{j}2} \nabla p^{I,j} + \cdots.
	\end{aligned}
	\right.
\end{empheq}

\subsection{Formal derivation of the fluid behavior far from the boundary}

We put all the asymptotic identities \eqref{eqs:AnsatzSyst} into the system \eqref{NSCeps} and we deduce that
\begin{empheq}{equation}
	\label{eq:Interior} \sum_{j=0}^1 \varepsilon^{\frac{j}2} \partial_t u^{I,j} - \sum_{j=0}^1 \varepsilon^{1+\frac{j}2} \Delta u^{I,j} + \sum_{j=0}^1 \varepsilon^{\frac{j}2} \sum_{k=0}^j u^{I,k} \cdot \nabla u^{I,j-k} + \sum_{j=0}^1 \varepsilon^{\frac{j}2-1} \begin{pmatrix} u^{I,j}_3\\0\\-u^{I,j}_1 \end{pmatrix} + \sum_{j=-2}^1 \varepsilon^{\frac{j}2} \nabla p^{I,j} = 0(\varepsilon).
\end{empheq}
Taking the limit $y = \frac{x_3}{\sqrt{\eps}} \to +\infty$ ($\eps \to 0$), the divergence-free property writes 
\begin{empheq}{equation}  
	\label{eq:divI} 
	\divv u^{I,j} = 0, \qquad \forall \; j\geq 0.
\end{empheq}

\noindent
\textbf{At the leading term of $\varepsilon^{-1}$} in \eqref{eq:Interior}, we simply have
\begin{empheq}{equation}
    \label{eq:minus1I} \begin{pmatrix} u^{I,0}_3\\0\\-u^{I,0}_1 \end{pmatrix} + \begin{pmatrix} \partial_1 p^{I,-2}\\ \partial_2 p^{I,-2}\\ \partial_3 p^{I,-2} \end{pmatrix} = 0.
\end{empheq}
Then, classical calculations (see Grenier-Masmoudi \cite{grenier-masmoudi} or Chemin {\it et al.} \cite{CDGGbook}) give
\begin{empheq}{equation}
    \label{eq:minus1IuA}  \partial_2 p^{I,-2} = \partial_2 u^{I,0}_1 = \partial_2 u^{I,0}_2 = \partial_2 u^{I,0}_3 = 0.
\end{empheq}

\noindent
\textbf{At the order $\varepsilon^{-1/2}$} in \eqref{eq:Interior}, we have
\begin{empheq}{equation}
    \label{eq:minus0.5I} \begin{pmatrix} u^{I,1}_3\\0\\-u^{I,1}_1 \end{pmatrix} + \begin{pmatrix} \partial_1 p^{I,-1}\\ \partial_2 p^{I,-1}\\ \partial_3 p^{I,-1} \end{pmatrix} = 0,
\end{empheq}
which imply
\begin{empheq}{equation}
	\label{eq:minus0.5Ip} \partial_2 p^{I,-1} =  \partial_2 u^{I,1}_1 = \partial_2 u^{I,1}_2 = \partial_2 u^{I,1}_3 = 0.
\end{empheq}

\begin{rem}
	Identities \eqref{eq:minus1IuA} and \eqref{eq:minus0.5Ip} mean that the limit behaviour of the outer flow is two-dimensional, as predicts the Taylor-Proudman theorem.
\end{rem}

\noindent
\textbf{At the order $\varepsilon^{0}$} in \eqref{eq:Interior}, taking into account \eqref{eq:minus1IuA} and the divergence-free condition \eqref{eq:divI}, we obtain
\begin{empheq}{equation}
	\label{eq:Ordre0Iter}
    \left\{
    \begin{aligned}
    	&\partial_t u^{I,0}_1 + u^{I,0}_1 \partial_1 u^{I,0}_1 + u^{I,0}_3 \partial_3 u^{I,0}_1 + \partial_1 p^{I,0} = 0\\
    	&\partial_t u^{I,0}_2 + u^{I,0}_1 \partial_1 u^{I,0}_2 + u^{I,0}_3 \partial_3 u^{I,0}_2 + \partial_2 p^{I,0}= 0\\
    	&\partial_t u^{I,0}_3 + u^{I,0}_1 \partial_1 u^{I,0}_3 + u^{I,0}_3 \partial_3 u^{I,0}_3 + \partial_3 p^{I,0} = 0\\
    	&\partial_2 u^{I,0}_1 = \partial_2 u^{I,0}_2 = \partial_2 u^{I,0}_3 = 0\\
    	&\partial_1 u^{I,0}_1 + \partial_3 u^{I,0}_3 = 0
    \end{aligned}
    \right.
\end{empheq}
Now, by applying  $\dd_2$  to the second equation of the system \eqref{eq:Ordre0Iter}, we obtain $$\dd_2^2 p^{I,0} = 0,$$ which means that there exist $g_1(x_1,x_3)$ and $g_2(x_1,x_3)$ such that 
$$
p^{I,0} = x_2 g_1 + g_2.
$$ 
Now, differentiating the first and third equations of \eqref{eq:Ordre0Iter} with respect to $x_2$, we obtain 
$$
\dd_1 g_1 = \dd_3 g_1 = 0.
$$ 
By taking $\abs{x} \to +\infty$ in the second equation of \eqref{eq:Ordre0Iter}, we conclude that $g_1 \equiv 0$. Thus, the system \eqref{eq:Ordre0Iter} becomes the following 2D Euler-type system with three components in the half-plane, which is the formal limiting system of \eqref{NSCeps} far from the boundary as $\eps \to 0$
$$
    \left\{
    \begin{aligned}
    	&\partial_t u^{I,0}_1 + u^{I,0}_1 \partial_1 u^{I,0}_1 + u^{I,0}_3 \partial_3 u^{I,0}_1 + \partial_1 p^{I,0} = 0\\
    	&\partial_t u^{I,0}_2 + u^{I,0}_1 \partial_1 u^{I,0}_2 + u^{I,0}_3 \partial_3 u^{I,0}_2 = 0\\
    	&\partial_t u^{I,0}_3 + u^{I,0}_1 \partial_1 u^{I,0}_3 + u^{I,0}_3 \partial_3 u^{I,0}_3 + \partial_3 p^{I,0} = 0\\
    	&\partial_2 u^{I,0}_1 = \partial_2 u^{I,0}_2 = \partial_2 u^{I,0}_3 = \dd_2 p^{I,0} = 0\\
    	&\partial_1 u^{I,0}_1 + \partial_3 u^{I,0}_3 = 0\\
    	&u^{I,0}_3|_{x_3=0} = 0.
    \end{aligned}
    \right.
$$
Since this system is independent of $x_2$, for the compatibility, we need to impose the  well prepared initial data, which means that $$u^{I,0}(0,x_1,x_3) = u^{I,0}_0(x_1,x_3).$$ The boundary condition will be discussed in \eqref{eq:ZerothBC}.

\noindent The system \eqref{eq:Ordre0Ibis} will be completed with a boundary condition for the second component $u^{I,0}_2$. In fact, the trace function $\overline{u^{I,0}_2}(t,x_1)$ on the boundary $\set{x_3 = 0}$ satisfies the following system
$$
	\left\{
    \begin{aligned}
    	&\partial_t\overline{u^{I,0}_2} + \overline{u^{I,0}_1} \partial_1 \overline{u^{I,0}_2} = 0\\
    	&\overline{u^{I,0}_2}(0,x_1) = u^{I,0}_{0,2}(x_1,0).
    \end{aligned}
	\right.
$$

\noindent
\textbf{At the order $\varepsilon^{1/2}$} in \eqref{eq:Interior}, using \eqref{eq:minus0.5Ip} and the divergence-free condition \eqref{eq:divI}, we obtain the system
\begin{empheq}{equation*}
    \left\{
    \begin{aligned}
		&\partial_t u^{I,1}_1 + u^{I,0}_1 \partial_1 u^{I,1}_1 + u^{I,0}_3 \partial_3 u^{I,1}_1 + u^{I,1}_1 \partial_1 u^{I,0}_1 + u^{I,1}_3 \partial_3 u^{I,0}_1 + \dd_1 p^{I,1} = 0\\
		&\partial_t u^{I,1}_2 + u^{I,0}_1 \partial_1 u^{I,1}_2 + u^{I,0}_3 \partial_3 u^{I,1}_2 + u^{I,1}_1 \partial_1 u^{I,0}_2 + u^{I,1}_3 \partial_3 u^{I,0}_2 + \dd_2 p^{I,1} = 0\\
		&\partial_t u^{I,1}_3 + u^{I,0}_1 \partial_1 u^{I,1}_3 + u^{I,0}_3 \partial_3 u^{I,1}_3 + u^{I,1}_1 \partial_1 u^{I,0}_3 + u^{I,1}_3 \partial_3 u^{I,0}_3 + \dd_3 p^{I,1} = 0\\
    	&\partial_2 u^{I,1}_1 = \partial_2 u^{I,1}_2 = \partial_2 u^{I,1}_3 = \dd_2 p^{I,1} = 0\\
    	&\partial_1 u^{I,1}_1 + \partial_3 u^{I,1}_3 = 0
    \end{aligned}
    \right.
\end{empheq}
We also remark that we can not obtain any determined boundary condition for $u^{I,1}$, but only a condition depending on the boundary condition of $u^{B,1}$. Indeed, on the boundary, we recall the value of $u^{I,j}_i$ is related to the value of $u^{B,j}_i$ by the equation
\begin{equation*}
	u^{I,j}_i(t,x_1,0) + u^{B,j}_i(t,x_1,0) = 0\qquad j=0, 1; \quad i=1, 2, 3.
\end{equation*}

Using the same argument, we can prove that $\dd_2 p^{I,1} = 0$, and we obtain the following 2D linearized Euler-type system with three components  in the half-plane
$$
    \left\{
    \begin{aligned}
		&\partial_t u^{I,1}_1 + u^{I,0}_1 \partial_1 u^{I,1}_1 + u^{I,0}_3 \partial_3 u^{I,1}_1 + u^{I,1}_1 \partial_1 u^{I,0}_1 + u^{I,1}_3 \partial_3 u^{I,0}_1 + \dd_1 p^{I,1} = 0\\
		&\partial_t u^{I,1}_2 + u^{I,0}_1 \partial_1 u^{I,1}_2 + u^{I,0}_3 \partial_3 u^{I,1}_2 + u^{I,1}_1 \partial_1 u^{I,0}_2 + u^{I,1}_3 \partial_3 u^{I,0}_2 = 0\\
		&\partial_t u^{I,1}_3 + u^{I,0}_1 \partial_1 u^{I,1}_3 + u^{I,0}_3 \partial_3 u^{I,1}_3 + u^{I,1}_1 \partial_1 u^{I,0}_3 + u^{I,1}_3 \partial_3 u^{I,0}_3 + \dd_3 p^{I,1} = 0\\
    	&\partial_2 u^{I,1}_1 = \partial_2 u^{I,1}_2 = \partial_2 u^{I,1}_3 = \dd_2 p^{I,1} = 0\\
    	&\partial_1 u^{I,1}_1 + \partial_3 u^{I,1}_3 = 0\\
        &u^{I,1}_3(t,x_1,0) = -u^{B,1}_3(t,x_1,0)\\
    	&u^{I,1}(0,x_1,x_3) = u^{I,1}_0(x_1,x_3).
    \end{aligned}
    \right.
$$
Here, we also suppose that the initial data are well prepared, \emph{i.e.} independent of $x_2$.

\subsection{Formal asymptotic expansions inside the boundary layer}\label{subsection2.2}

Inside the boundary layer (in the domain $0<x_3\le \sqrt{\varepsilon}$), we consider the Taylor expansions
\begin{empheq}{align*}
	u^{I,j}_i(t,x_h,x_3) &= u^{I,j}_i(t,x_h,0) + x_3 \partial_3 u^{I,j}_i(t,x_h,0) + \frac{x_3^2}2 \partial_3^2 u^{I,j}_i(t,x_h,0) + \ldots\\
	p^{I,j}(t,x_h,x_3) &= p^{I,j}(t,x_h,0) + x_3 \partial_3 p^{I,j}(t,x_h,0) + \frac{x_3^2}2 \partial_3^2 p^{I,j}(t,x_h,0) + \ldots
\end{empheq}
Performing the change of variable $y = \frac{x_3}{\sqrt{\varepsilon}}$, we have
\begin{empheq}{equation}
	\label{eq:traceUI}
	\left\{
	\begin{aligned}
		u^{I,j}_i(t,x_h,x_3) &= \overline{u^{I,j}_i} + \varepsilon^{\frac{1}2} y \overline{\partial_3 u^{I,j}_i} + \frac{\varepsilon y^2}{2!}  \overline{\partial_3^2 u^{I,j}_i} + \cdots\\
		p^{I,j}(t,x_h,x_3) &= \overline{p^{I,j}} + \varepsilon^{\frac{1}2} y \overline{\partial_3 p^{I,j}} + \frac{\varepsilon y^2}{2!} \overline{\partial_3^2 p^{I,j}} + \cdots
	\end{aligned}
	\right.
\end{empheq}
where $\overline{f} = f(t,x_1,x_2,0)$ is the trace of $f$ on $\set{x_3 = 0}$. 
Now, we will rewrite the identities \eqref{eqs:AnsatzSyst}, taking into account the expansion \eqref{eq:traceUI}. First, we have
\begin{empheq}{align}
	\label{eqs:AnsatzBound00} 
u^\varepsilon &= \pare{u^{B,0} + \overline{u^{I,0}}} + \varepsilon^{\frac{1}2} \pare{u^{B,1} + \overline{u^{I,1}} + y\overline{\partial_3 u^{I,0}}} + \sum_{k=2}^3 \varepsilon^{\frac{k}2} \pare{\frac{y^{k-1}}{(k-1)!} \overline{\partial_3^{k-1} u^{I,1}} + \frac{y^k}{k!} \overline{\partial_3^k u^{I,0}}} + \cdots \\
	&= {\Ucal^{p,0} + \varepsilon^{\frac{1}2} \Ucal^{p,1} + \sum_{k=2}^3 \varepsilon^{\frac{k}2} \pare{\frac{y^{k-1}}{(k-1)!} \overline{\partial_3^{k-1} u^{I,1}} + \frac{y^k}{k!} \overline{\partial_3^k u^{I,0}}} +\cdots.} \notag
\end{empheq}
where we note 
\begin{equation}\label{eqs:AnsatzBounddt}
\Ucal^{p,0}=u^{B,0} + \overline{u^{I,0}},\qquad  \Ucal^{p,1}=u^{B,1} + \overline{u^{I,1}} + y\overline{\partial_3 u^{I,0}}.
\end{equation}
The derivatives of $u^\varepsilon$ with respect to tangential variables write
$$
	 \partial^m_{t, 1, 2} u^\varepsilon =  \partial^m_{t, 1, 2} \Ucal^{p,0} + \varepsilon^{\frac{1}2} \partial^m_{t, 1, 2} \Ucal^{p,1} + \sum_{k=2}^3 \varepsilon^{\frac{k}2} \partial^m_{t, 1, 2} \pare{\frac{y^{k-1}}{(k-1)!} \overline{\partial_3^{k-1} u^{I,1}} + \frac{y^k}{k!} \overline{\partial_3^k u^{I,0}}} + \cdots.
$$
where $m=1, 2$.  For the normal variable, we have
\begin{equation*}
	\partial_3 u^\varepsilon
	= \varepsilon^{-\frac{1}2} \partial_y u^{B,0} + \pare{\partial_y u^{B,1} + \overline{\partial_3 u^{I,0}}} + \sum_{k=1}^3 \varepsilon^{\frac{k}{2}} \pare{\frac{y^{k-1}}{(k-1)!} \overline{\partial_3^{k} u^{I,1}} + \frac{y^k}{k!} \overline{\partial_3^{k+1} u^{I,0}}} + \cdots
\end{equation*}
and
$$
	\partial_3^2 u^\varepsilon = \varepsilon^{-1} \partial_y^2 u^{B,0} + \varepsilon^{-\frac{1}{2}} \partial_y^2 u^{B,1} + \overline{\partial_3^2 u^{I,0}} + \sum_{k=1}^3 \varepsilon^{\frac{k}{2}} \pare{\frac{y^{k-1}}{(k-1)!} \overline{\partial_3^{k+1} u^{I,1}} + \frac{y^k}{k!} \overline{\partial_3^{k+2} u^{I,0}}} + \cdots.
$$
Thus,
\begin{align*}
	-\varepsilon \Delta u^\varepsilon &= -\varepsilon \Delta_h \Ucal^{p,0} - \varepsilon^{\frac{3}2} \Delta_h \Ucal^{p,1}-\partial_y^2 \Ucal^{p,0} - \varepsilon^{\frac{1}2} \partial_y^2 \Ucal^{p,1} - \varepsilon \overline{\partial_3^2 u^{I,0}}\\
&\qquad\qquad - \varepsilon^{\frac{3}2} \pare{\frac{y^{k-1}}{(k-1)!} \overline{\partial_3^{k+1} u^{I,1}} + \frac{y^{k}}{k!} \overline{\partial_3^{k+2} u^{I,0}}} + \cdots .\notag
\end{align*}

For the non-linear term, we only give the explicit calculations for the first orders of its expansion. We write
$$
u^\varepsilon\cdot\nabla u^\varepsilon = u^\varepsilon_h\cdot \nabla_h u^\varepsilon + u^\varepsilon_3 \partial_3 u^\varepsilon.
$$
Then,  we have
\begin{align*}
	u^\varepsilon_h\cdot \nabla_h u^\varepsilon&=\Ucal^{p,0}_h \cdot \nabla_h \Ucal^{p,0}_h + \varepsilon^{\frac{1}2} \Ucal^{p,0}_h \cdot \nabla_h \Ucal^{p,1} + \varepsilon^{\frac{1}2} \Ucal^{p,1}_h \cdot \nabla_h \Ucal^{p,0}_h + \cdots\\	
	u^\varepsilon_3 \partial_3 u^\varepsilon
	&=\Ucal^{p,1}_3 \partial_y \Ucal^{p,0}_h + \varepsilon^{\frac{1}{2}} \Ucal^{p,1}_3 \partial_y \Ucal^{p,1} + \varepsilon^{\frac{1}2} \pare{y\overline{\partial_3 u^{I,1}_3} + \frac{y^2}{2} \overline{\partial_3^2 u^{I,0}_3}} \partial_y \Ucal^{p,0}_h + \cdots .
\end{align*}

\noindent For the Coriolis forcing term (the rotation term), we have
\begin{align*}
	&\frac{e_2\times u^\varepsilon}{\varepsilon}\\
	&=\varepsilon^{-1} \begin{pmatrix} 0\\0\\-\Ucal^{p,0}_1 \end{pmatrix} + \varepsilon^{-\frac{1}2} \begin{pmatrix} \Ucal^{p,1}_3\\0\\-\Ucal^{p,1}_1 \end{pmatrix} + \sum_{k=2}^3 \varepsilon^{\frac{k}2-1} \pint{\frac{y^{k-1}}{(k-1)!} \begin{pmatrix} \overline{\partial_3^{k-1} u^{I,1}_3} \\0\\-\overline{\partial_3^{k-1} u^{I,1}_1} \end{pmatrix} + \frac{y^k}{k!} \begin{pmatrix} \overline{\partial_3^k u^{I,0}_3} \\0\\-\overline{\partial_3^k u^{I,0}_1} \end{pmatrix} } + \cdots. \notag
\end{align*}
Finally, the pressure term is
\begin{empheq}{align}
	\label{eqs:AnsatzBound05}
	\begin{pmatrix} \dd_{x_1} p^\varepsilon \\ \dd_{x_2} p^\varepsilon\\ \dd_{x_3} p^\varepsilon \end{pmatrix}  &= \varepsilon^{-\frac{3}2} \begin{pmatrix} 0\\0\\ \partial_y \Pcal^{p,-2} \end{pmatrix} + \sum_{j=-2}^{-1} \varepsilon^{\frac{j}2} \begin{pmatrix} \partial_1 \Pcal^{p,j} \\ \partial_2 \Pcal^{p,j} \\ \partial_y \Pcal^{p,j+1} \end{pmatrix} + \pint{\begin{pmatrix} \partial_1 \Pcal^{p,0} \\ \partial_2 \Pcal^{p,0} \\ 0 \end{pmatrix} + \sum_{k=0}^2 \frac{y^k}{k!} \begin{pmatrix} 0\\0\\ \overline{\partial_3^{k+1} p^{I,-k}} \end{pmatrix}} \\
	&\qquad + \varepsilon^{\frac{1}2} \sum_{j=-2}^1 \frac{y^{1-j}}{(1-j)!} \begin{pmatrix} \overline{\partial_1 \partial_3^{1-j} p^{I,j}}\\ \overline{\partial_2 \partial_3^{1-j} p^{I,j}}\\ \overline{\partial_3^{2-j} p^{I,j}} \end{pmatrix} + \cdots. \notag
\end{empheq}
where
\begin{empheq}{align}
	\label{eq:Pprandtl2} &\Pcal^{p,-2} = p^{B,-2} + \overline{p^{I,-2}},\qquad \Pcal^{p,-1} = p^{B,-1} + \overline{p^{I,-1}} + y\overline{\partial_3 p^{I,-2}}\\
	\label{eq:Pprandtl0} &\Pcal^{p,0} = p^{B,0} + \overline{p^{I,0}} + y\overline{\partial_3 p^{I,-1}} + \frac{y}2\overline{\partial_3 p^{I,-2}}.
\end{empheq}

\subsection{Incompressibility and boundary conditions} 
The divergence-free property of the velocity field is rewritten as follows
\begin{empheq}{align*}
	0 &= \divv u^\varepsilon = \varepsilon^{-\frac{1}2} \partial_y u^{B,0}_3\pare{t,x_h,\frac{x_3}{\sqrt{\varepsilon}}} + \pint{\divv u^{I,0} + \partial_1 u^{B,0}_1 + \partial_2 u^{B,0}_2 + \partial_y u^{B,1}_3\pare{t,x_h,\frac{x_3}{\sqrt{\varepsilon}}}}\\
	& \qquad \qquad \qquad \qquad \qquad \qquad \qquad \qquad + \varepsilon^{\frac{1}2} \pint{\divv u^{I,1} + \partial_1 u^{B,1}_1 + \partial_2 u^{B,1}_2} + \cdots.
\end{empheq}
Inside the boundary layer, using the expansion \eqref{eq:divI} and \eqref{eq:traceUI}, we deduce the following divergence-free condition
\begin{empheq}{equation*}
	\varepsilon^{-\frac{1}2} \partial_y u^{B,0}_3 + \pare{\partial_1 u^{B,0}_1 + \partial_2 u^{B,0}_2 + \partial_y u^{B,1}_3} + \varepsilon^{\frac{1}2} \pare{\partial_1 u^{B,1}_1 + \partial_2 u^{B,1}_2}= 0.
\end{empheq}
Thus, we obtain the incompressibility of the boundary layer
\begin{empheq}{align}
    \label{eq:divV0b} &\partial_1 u^{B,0}_1 + \partial_2 u^{B,0}_2 + \partial_y u^{B,1}_3 = 0,\\  
    &\partial_1 u^{B,1}_1 + \partial_2 u^{B,1}_2 = 0.\nonumber
\end{empheq}
Moreover, we have
\begin{equation*}	
\partial_y u^{B,0}_3 = 0,
\end{equation*}
which, by taking $y\to +\infty$, gives
$$
	\label{eq:ub03} u^{B,0}_3 = 0.
$$
For the boundary condition in \eqref{NSCeps} on $\set{x_3 = 0}$, we have
$$
	\label{eq:boundarycond} \sum_{j=0}^1 \varepsilon^{\frac{j}2} \pint{u^{I,j} (t,x_h,0) + u^{B,j} (t,x_h,0)} = 0,
$$
which implies that
$$
 \overline{u^{I,0}(t)} + u^{B,0} (t,x_h,0) = 0,
$$
\begin{empheq}{align}
	\label{eq:bound1} \overline{u^{I,1}(t)} + u^{B,1} (t,x_h,0) = 0.
\end{empheq}
In particular, $u^{B,0}_3 = 0$ imply
\begin{equation}\label{eq:ZerothBC}
 u^{I,0}_3|_{x_3=0}= \overline{u^{I,0}_3} = 0,
\end{equation}
which is the boundary condition for Euler equation in \eqref{eq:Ordre0Ibis}, and the third components in \eqref{eq:bound1} gives the boundary condtion of linearized Euler equation in \eqref{eq:Ordre0.5I}.

\subsection{Formal derivation of the governing equations of the fluid in the boundary layer}

Now, we consider the system \eqref{NSCeps} near $\set{x_3 = 0}$, using the asymptotic formal \eqref{eqs:AnsatzBound00} - \eqref{eqs:AnsatzBound05}. 

\noindent
\textbf{At the order $\varepsilon^{-\frac{3}{2}}$}, we have
\begin{empheq}{equation*}  
\partial_y p^{B,-2} = 0,
\end{empheq}
which implies that $p^{B,-2} = 0$ because $p^{B,-2}$ goes to zero as $y\to +\infty$. Using the new notation of the pressure defined in \eqref{eq:Pprandtl2}, we get
$$ 
\partial_y \Pcal^{p,-2} = 0.
$$

\noindent
\textbf{At the order $\varepsilon^{-1}$}, using the fact that $\overline{u^{I,0}_3} = 0$, $u^{B,0}_3 = 0$ and $p^{B,-2} = 0$, we get
\begin{empheq}{equation*}
    \label{eq:minus1} \begin{pmatrix} 0\\0\\-u^{B,0}_1 \end{pmatrix} + \begin{pmatrix} 0\\0\\-\overline{u^{I,0}_1} \end{pmatrix} + \begin{pmatrix} 0\\0\\ \partial_y p^{B,-1} \end{pmatrix} + \overline{\nabla p^{I,-2}} = 0,
\end{empheq}
which implies that $\overline{\partial_1 p^{I,-1}} = \overline{\partial_2 p^{I,-1}} = 0$ and
\begin{empheq}{equation}
	\label{eq:minus1B} - u^{B,0}_1 - \overline{u^{I,0}_1} + \partial_y p^{B,-1} + \overline{\partial_3 p^{I,-2}} = 0.
\end{empheq}
Using the new velocity and pressure defined in \eqref{eq:Pprandtl2} and taking into account the fact that $\Ucal^{p,0}_3 = 0$, we can also write\begin{empheq}{equation}
	\label{eq:Pminus1B} \begin{pmatrix} 0\\0\\-\Ucal^{p,0}_1 \end{pmatrix} + \begin{pmatrix} \partial_1\Pcal^{p,-2}\\ \partial_2 \Pcal^{p,-2}\\ \partial_y \Pcal^{p,-1}\end{pmatrix} = 0.
\end{empheq}

\noindent
\textbf{At the order $\varepsilon^{-1/2}$}, we have
\begin{empheq}{equation*}
    \begin{pmatrix} u^{B,1}_3\\0\\-u^{B,1}_1 \end{pmatrix} + \begin{pmatrix} \overline{u^{I,1}_3} \\0\\-\overline{u^{I,1}_1} \end{pmatrix} + y \begin{pmatrix} \overline{\partial_3 u^{I,0}_3} \\0\\-\overline{\partial_3 u^{I,0}_1} \end{pmatrix} + \begin{pmatrix} \partial_1 p^{B,-1} \\ \partial_2 p^{B,-1} \\ \partial_y p^{B,0} \end{pmatrix} + \begin{pmatrix} \overline{\partial_1 p^{I,-1}}\\ \overline{\partial_2 p^{I,-1}}\\ \overline{\partial_3 p^{I,-1}} \end{pmatrix} + y \begin{pmatrix} \overline{\partial_1 \partial_3 p^{I,-2}}\\ \overline{\partial_2 \partial_3 p^{I,-2}}\\ \overline{\partial_3^2 p^{I,-2}} \end{pmatrix} = 0,
\end{empheq}
or in a equivalent way, using the new velocity and pressure defined in \eqref{eq:Pprandtl0},
\begin{empheq}{equation}
    \label{eq:Pminus0.5B} \begin{pmatrix} \Ucal^{p,1}_3\\0\\-\Ucal^{p,1}_1 \end{pmatrix} + \begin{pmatrix} \partial_1 \Pcal^{p,-1} \\ \partial_2 \Pcal^{p,-1} \\ \partial_y \Pcal^{p,0} \end{pmatrix} = 0.
\end{empheq}
then 
$$
\partial_2 \Pcal^{p,-1} = 0.
$$ 
and
\begin{empheq}{align*}
	&\partial_2 \Ucal^{p,0}_1 = \partial_2 \partial_y \Pcal^{p,-1} = \partial_y \partial_2 \Pcal^{p,-1} = 0\\
	&\partial_2 \Ucal^{p,1}_3 = -\partial_2 \partial_1 \Pcal^{p,-1} = -\partial_1 \partial_2 \Pcal^{p,-1} = 0.
\end{empheq}
Using the divergence-free properties \eqref{eq:divI} and \eqref{eq:divV0b}, we also have
\begin{empheq}{equation*}
	\partial_2 \Ucal^{p,0}_2 = - \partial_1 \Ucal^{p,0}_1 - \partial_y \Ucal^{p,1}_3 = - \partial_1 \partial_y \Pcal^{p,-1} - \pare{-\partial_y \partial_1 \Pcal^{p,-1}} = 0.
\end{empheq}
We deduce that $(\Ucal^{p,0}_1,\Ucal^{p,0}_2,\Ucal^{p,1}_3)$ is a divergence-free vector field which is independent on $x_2$. The fact that $\dd_2 u^{I,0} = \dd_2 u^{I,1} = 0$ implies that
\begin{equation}\label{c2}
\dd_2 u^{B,0}_1 = \dd_2 u^{B,0}_2 = \dd_2 u^{B,1}_3 = 0.
\end{equation}

\begin{rem}
	The leading order of the velocity of the fluid inside the boundary layer also obeys the Taylor-Proudman theorem.
\end{rem}

\noindent
\textbf{At the order $\varepsilon^{0}$}, recalling that $u^{B,0}_3 = \overline{u^{I,0}_3} = 0$, we get the following equation
\begin{empheq}{multline*}
    \partial_t \pare{u^{B,0}_h + \overline{u^{I,0}_h}} - \partial_y^2 u^{B,0}_h + \pare{u^{B,0}_h + \overline{u^{I,0}_h}}\cdot \nabla_h \pare{u^{B,0}_h + \overline{u^{I,0}_h}} + \pare{u^{B,1}_3 + \overline{u^{I,1}_3} + y\overline{\partial_3 u^{I,0}_3}} \partial_y u^{B,0}_h\\
    + y \begin{pmatrix} \overline{\partial_3 u^{I,1}_3} \\0\\-\overline{\partial_3 u^{I,1}_1} \end{pmatrix} + \frac{y^2}2 \begin{pmatrix} \overline{\partial_3^2 u^{I,0}_3} \\0\\-\overline{\partial_3^2 u^{I,0}_1} \end{pmatrix} + \begin{pmatrix} \partial_1 p^{B,0} \\ \partial_2 p^{B,0} \\ 0 \end{pmatrix} + \begin{pmatrix} \overline{\partial_1 p^{I,0}}\\ \overline{\partial_2 p^{I,0}}\\ \overline{\partial_3 p^{I,0}} \end{pmatrix} + y \begin{pmatrix} \overline{\partial_1 \partial_3 p^{I,-1}}\\ \overline{\partial_2 \partial_3 p^{I,-1}}\\ \overline{\partial_3^2 p^{I,-1}} \end{pmatrix} + \frac{y^2}{2} \begin{pmatrix} \overline{\partial_1 \partial_3^2 p^{I,-2}}\\ \overline{\partial_2 \partial_3^2 p^{I,-2}}\\ \overline{\partial_3^3 p^{I,-2}} \end{pmatrix} = 0.
\end{empheq}
From \eqref{eq:minus1I} and \eqref{eq:minus0.5I}, we deduce that
$$
-y \overline{\partial_3 u^{I,1}_1} - \frac{y^2}2 \overline{\partial_3^2 u^{I,0}_1} + y \overline{\partial_3^2 p^{I,-1}} + \frac{y^2}2 \overline{\partial_3^3 p^{I,-2}} = 0.
$$
We also remark that the boundary condition applying to the third equation of the Euler system implies that $$\overline{\partial_3 p^{I,0}} = 0.$$ Then, using the new velocity and pressure defined in \eqref{eqs:AnsatzBounddt} and \eqref{eq:Pprandtl2}, we get
\begin{empheq}{multline*}
	\partial_t \Ucal^{p,0}_h - \partial_y^2 \Ucal^{p,0}_h + \Ucal^{p,0}_h\cdot\nabla_h \Ucal^{p,0}_h + \Ucal^{p,1}_3 \partial_y \Ucal^{p,0}_h + \begin{pmatrix} \partial_1 p^{B,0} + \overline{\partial_1 p^{I,0}} \\ \partial_2 \Pcal^{p,0}  \end{pmatrix} = 0.
\end{empheq}
Taking into account the divergence-free condition \eqref{eq:divV0b}, the identities \eqref{eq:minus1B} and \eqref{eq:Pminus1B}, and 
 $(\Ucal^{p,0}_1,\Ucal^{p,0}_2,\Ucal^{p,1}_3)$ is independs on $x_2$, we deduce that $(\Ucal^{p,0}_1,\Ucal^{p,0}_2,\Ucal^{p,1}_3)$ satisfies the following system
$$
    \left\{
    \begin{aligned}
    	&\partial_t \Ucal^{p,0}_1 - \partial_y^2 \Ucal^{p,0}_1 + \Ucal^{p,0}_1 \partial_1 \Ucal^{p,0}_1 + \Ucal^{p,1}_3 \partial_y \Ucal^{p,0}_1 + \partial_1 p^{B,0} + \overline{\partial_1 p^{I,0}} = 0\\
    	&\partial_t \Ucal^{p,0}_2 - \partial_y^2 \Ucal^{p,0}_2 + \Ucal^{p,0}_1 \partial_1 \Ucal^{p,0}_2 + \Ucal^{p,1}_3 \partial_y \Ucal^{p,0}_2 + \partial_2 \Pcal^{p,0} = 0\\
    	&\partial_1 \Ucal^{p,0}_1 + \partial_y \Ucal^{p,1}_3 = 0,\\
    	&\partial_2 \Ucal^{p,0}_1 = \partial_2 \Ucal^{p,0}_2 = \partial_2 \Ucal^{p,1}_3 = 0.
    \end{aligned}
	\right.
$$
We remark that the above system is not complete, since we need another equation for the component $\Ucal^{p,1}_3$.

\noindent
\textbf{At the order $\varepsilon^{1/2}$}, we have
\begin{empheq}{multline*}
    \partial_t \Ucal^{p,1} - \partial_y^2 \Ucal^{p,1} + \Ucal^{p,0}_h \cdot \nabla_h \Ucal^{p,1} + \Ucal^{p,1}_h \cdot \nabla_h \Ucal^{p,0}_h + \Ucal^{p,1}_3 \partial_y \Ucal^{p,1} + \pare{y\overline{\partial_3 u^{I,1}_3} + \frac{y^2}{2} \overline{\partial_3^2 u^{I,0}_3}} \partial_y \Ucal^{p,0}_h\\
    + \pint{\frac{y^2}{2} \begin{pmatrix} \overline{\partial_3^2 u^{I,1}_3} \\0\\-\overline{\partial_3^2 u^{I,1}_1} \end{pmatrix} + \frac{y^3}{6} \begin{pmatrix} \overline{\partial_3^3 u^{I,0}_3} \\0\\-\overline{\partial_3^3 u^{I,0}_1} \end{pmatrix} } + \sum_{k=0}^3 \frac{y^k}{k!} \begin{pmatrix} \overline{\partial_1 \partial_3^k p^{I,1-k}}\\ \overline{\partial_2 \partial_3^k p^{I,1-k}}\\ \overline{\partial_3^{k+1} p^{I,1-k}} \end{pmatrix} = 0.
\end{empheq}
Here, we are only interested in the component $\Ucal^{p,1}_3$. Using the fact that $\partial_2 \Ucal^{p,1}_3 = 0$, we obtain
$$ 
\partial_t \Ucal^{p,1}_3 - \partial_y^2 \Ucal^{p,1}_3 + \Ucal^{p,0}_1 \partial_1 \Ucal^{p,1}_3 + \Ucal^{p,1}_3 \partial_y \Ucal^{p,1}_3 + \overline{\partial_3 p^{I,1}} + y\overline{\partial_3^2 p^{I,0}} = 0.
$$

Collect all the above formal calculations, we deduce the following governing equations of the boundary layer
\begin{empheq}{equation}
    \label{eq:PrRotE2P1} \tag{P1}
    \left\{
    \begin{aligned}
    	&\partial_t \Ucal^{p,0}_1 - \partial_y^2 \Ucal^{p,0}_1 + \Ucal^{p,0}_1 \partial_1 \Ucal^{p,0}_1 + \Ucal^{p,1}_3 \partial_y \Ucal^{p,0}_1 + \partial_1 p^{B,0} + \overline{\partial_1 p^{I,0}} = 0\\
    	&\partial_t \Ucal^{p,1}_3 - \partial_y^2 \Ucal^{p,1}_3 + \Ucal^{p,0}_1 \partial_1 \Ucal^{p,1}_3 + \Ucal^{p,1}_3 \partial_y \Ucal^{p,1}_3 + \overline{\partial_3 p^{I,1}} + y\overline{\partial_3^2 p^{I,0}} = 0\\
    	&\partial_1 \Ucal^{p,0}_1 + \partial_y \Ucal^{p,1}_3 = 0\\
    	&\partial_2 \Ucal^{p,0}_1 = \partial_2 \Ucal^{p,1}_3 = 0\\
    	&\Ucal^{p,0}_1(t,x_1,0) = 0, \quad \lim_{y\to +\infty} \Ucal^{p,0}_1(t,x_1,y) = \overline{u^{I,0}_1}(x_1)\\
    	&\Ucal^{p,1}_3(t,x_1,0) = 0, \quad \partial_y \Ucal^{p,1}_3(t,x_1,0) = 0\\
       &\Ucal^{p,0}_1(0,x_1,y) = u^{B,0}_{0,1} (x_1,y) + \overline{u^{I,0}_{0,1}}(x_1)\\
    	&\Ucal^{p,1}_3(0,x_1,y) = u^{B,1}_{0,3} (x_1,y) + \overline{u^{I,1}_{0,3}}(x_1) + y\overline{\dd_3u^{I,0}_{0,3}}(x_1).
    \end{aligned}
	\right.
\end{empheq}
and
\begin{empheq}{equation}
	\label{eq:P2}\tag{P2}
	\left\{
	\begin{aligned}
		&\partial_t \Ucal^{p,0}_2 - \partial_y^2 \Ucal^{p,0}_2 + \Ucal^{p,0}_1 \partial_1 \Ucal^{p,0}_2 + \Ucal^{p,1}_3 \partial_y \Ucal^{p,0}_2 + \partial_2 \Pcal^{p,0} = 0\\
		&\Ucal^{p,0}_2(0,x_1,y) = u^{B,0}_{0,2} (x_1,y) + \overline{u^{I,0}_{0,2}}(x_1)\\
		&\Ucal^{p,0}_2(t,x_1,0) = 0,\quad \lim_{y\to +\infty} \Ucal^{p,0}_2(t,x_1,y) = \overline{u^{I,0}_2}(x_1)\\
               &\Ucal^{p,0}_2(0,x_1,y) = u^{B,0}_{0,2} (x_1,y) + \overline{u^{I,0}_{0,2}}(x_1).
	\end{aligned}
	\right.
\end{empheq}

\bigskip
\noindent \textbf{Claim:} {\it The pressure term of the \eqref{eq:P2} satisfies $\partial_2\Pcal^{p,0} = 0$.}

Indeed, applying $\partial_2$ to the first equation of the systems \eqref{eq:PrRotE2P1} and \eqref{eq:P2}, and using the fact that $$\partial_2\Ucal^{p,0}_1 = \partial_2\Ucal^{p,0}_2 = \partial_2\Ucal^{p,1}_3 = 0,$$ we deduce that 
$$
\partial_1\partial_2\Pcal^{p,0} = \partial_2^2\Pcal^{p,0} = 0.
$$ 
This means that, modulo a contant, we have
$$ 
\Pcal^{p,0} = x_2 G_1(t,y) + \int_{-\infty}^{x_1} \tilde{f}(t,x,y)dx,
$$
where 
$$
 G_1 = - \pare{\partial_t \Ucal^{p,0}_2 - \partial_y^2 \Ucal^{p,0}_2 + \Ucal^{p,0}_1 \partial_1 \Ucal^{p,0}_2 + \Ucal^{p,1}_3 \partial_y \Ucal^{p,0}_2}
$$ 
is to be determined and
\begin{equation*}
	\tilde{f} = \partial_1 \Pcal^{p,0} = -\partial_t \Ucal^{p,0}_1 + \partial_y^2 \Ucal^{p,0}_1 - \Ucal^{p,0}_1 \partial_1 \Ucal^{p,0}_1 - \Ucal^{p,1}_3 \partial_y \Ucal^{p,0}_1 - \pare{\frac{y^2}{2} \overline{\partial_3^2 u^{I,0}_3} + y\overline{\partial_3 u^{I,1}_3}}.
\end{equation*}

We recall that, from \eqref{eq:Pminus0.5B}, we have
$$\partial_y \Pcal^{p,0} = \Ucal^{p,1}_1,$$
where $\Ucal^{p,1}_1$ is the solution of the system
\begin{equation*}
    \left\{
    \begin{aligned}
    	&\partial_t \Ucal^{p,1}_1 - \partial_y^2 \Ucal^{p,1}_1 + \Ucal^{p,0}_1 \cdot \dd_1 \Ucal^{p,1}_1 + \Ucal^{p,0}_2 \cdot \dd_2 \Ucal^{p,1}_1 + \Ucal^{p,1}_3 \partial_y \Ucal^{p,1}_1 + \Ucal^{p,1}_1 \partial_1 \Ucal^{p,0}_1\\
    	&\qquad\qquad + \pare{y\overline{\partial_3 u^{I,1}_3} + \frac{y^2}{2} \overline{\partial_3^2 u^{I,0}_3}} \partial_y \Ucal^{p,0}_1 + \pint{\frac{y^2}{2} \overline{\partial_3^2 u^{I,1}_3} + \frac{y^3}{6} \overline{\partial_3^3 u^{I,0}_3} } + \sum_{k=0}^3 \frac{y^k}{k!} \overline{\partial_1 \partial_3^k p^{I,1-k}} = 0\\
    	&\Ucal^{p,1}_1(0,x_1,x_2,y) = u^{B,1}_{0,1}(x_1,y) + \overline{u^{I,1}_{0,1}}(x_1) + y\overline{\partial_3u^{I,0}_{0,1}}(x_1) + \alpha_1(y)x_2\\
    	&\Ucal^{p,1}_1(t,x_1,0) = 0.
    \end{aligned}
	\right.
\end{equation*}
We remark that $\dd_yG_1(t,y) = \dd_2 \Ucal^{p,1}_1$ and we recall that $\dd_1\dd_2 \Ucal^{p,1}_1 = \dd_2^2 \Ucal^{p,1}_1$. So, in fact, we will find $\dd_y G_1$ by solving the following system
\begin{equation}
	\label{eq:G1}
	\left\{
	\begin{aligned}
		&\partial_t(\partial_y G_1) - \partial_y^2 (\partial_y G_1) + (\partial_1 \Ucal^{p,0}_1)(\partial_y G_1) + \Ucal^{p,1}_3 \dd_y (\dd_y G_1) = 0\\
		&\dd_yG_1(0,y) = \alpha_1(y)\\
		&\dd_yG_1(t,0) = 0.
	\end{aligned}
	\right.
\end{equation}
where $\alpha_1$ is a given function, with $\alpha_1(0) = 0$. For the case of well prepared data, we consider the initial data to be independent of $x_2$, so $\alpha_1 \equiv 0$ and it is easy to see that the system \eqref{eq:G1} admits $0$ as a trivial solution. Then, the uniqueness of this solution implies $\dd_yG_1(t,.) \equiv 0$. Replacing $y=0$ in \eqref{eq:G1}, we obtain $G_1(t,0) = 0$, and so $G_1(t,.) \equiv 0$, for any $t \in \RR_+$.

\section{Well-posedness of the boundary layer system}\label{section3}

 In this section  we will prove the well-posedness for system \eqref{eq:PrRotE2}.  Since the pressure term in the first equation of  \eqref{eq:PrRotE2} is unknown, we begin with handling the second one to prove the existence of $ \Ucal^{p,1}_3$  and then  use the divergence-free property to find $\Ucal^{p,0}_1$ .  To do so we insert the representations 
 \begin{eqnarray*}
 \Ucal^{p,0}_1= u^{B,0}_1+ \overline{u^{I,0}_1}, \quad
 	\Ucal^{p,1}_3 = u^{B,1}_3 + \overline{u^{I,1}_3} + y \overline{\partial_3 u^{I,0}_3}
 \end{eqnarray*}
 into the second equation of  \eqref{eq:PrRotE2}, and then make use of the equations  \eqref{eq:Ordre0Ibis} and \eqref{eq:Ordre0.5I}  of  $u^{I,0}_3$ and $u^{I,1}_3$.  It then follows that the unknowns $u^{B,1}_3, u^{B,0}_1$ and $u^{I,1}_3$ satisfy the equation
 \begin{eqnarray*}
 	&&\pare{\partial_t - \partial_y^2 + \overline{\partial_3 u^{I,0}_3} y\partial_y} u^{B,1}_3  + \pare{ u^{B,0}_1+ \overline{u^{I,0}_1}} \partial_1 u^{B,1}_3\\
		&&\qquad  \qquad + \pare{u^{B,1}_3 +\overline{u^{I,1}_3} } \partial_y u^{B,1}_3+ \overline{\partial_3 u^{I,0}_3} u^{B,1}_3+ \pare{ -\partial_1\overline{u^{I,1}_3} + y\overline{\partial_1\partial_3 u^{I,0}_3}}  u^{B,0}_1= 0,
 \end{eqnarray*}
  and  the divergence-free properties \eqref{eq:divV0b} and \eqref{c2} yield 
 \begin{eqnarray*}
 	 u^{B,0}_1= - \int_{-\infty}^{x_1}\partial_y    u^{B,1}_3(t, z,y)dz.
 \end{eqnarray*}  
 Thus the above is just a equation for $u^{B,1}_3$. To solve the system \eqref{eq:PrRotE2}, we consider the following nonlinear initial-boundary problem, 
\begin{equation}
	\label{linsys}
		\left\{
	\begin{aligned}
		&\pare{\partial_t - \partial_y^2 + \overline{\partial_3 u^{I,0}_3} y\partial_y} u  + \pare{  v+ \overline{u^{I,0}_1}} \partial_1 u\\
		&\qquad  \qquad + \pare{ u -   u(t,x_1,0)} \partial_y u+ \overline{\partial_3 u^{I,0}_3} u+ \pare{\partial_1   u(t,x_1,0) + y\overline{\partial_1\partial_3 u^{I,0}_3}}  v= 0,\\
		&\partial_y u|_{y=0} =- \overline{\partial_3u^{I, 0}_3}(t,x_1),\quad\lim_{y\to +\infty} u(t,x_1,y) = 0,\\
		&u|_{t=0}= u_{0}(x_1,y),
	\end{aligned}
	\right.
\end{equation}
where  the unknown functions  $  u$ and $  v$ are  linked by the relation   
\begin{eqnarray}\label{relauv}
 v(t,x_1,y)=  - \int_{-\infty}^{x_1}\partial_y   u(t, z,y)dz.
\end{eqnarray}
Recall the functions $u^{I,0}_1, u^{I,0}_3$ are the solutions to the Euler-type system \eqref{eq:Ordre0Ibis}.  By Theorem \ref{th:E2D}, we see $u^{I,0}_1, u^{I,0}_3\in \mathcal A_\tau$ for some $\tau>0.$

The main result of this section can be stated as follows.

\begin{thm}\label{th41}
Suppose the initial data $u_{0}\in X_{\rho_0, a_0}$ for some $\rho_0>0$ and $a_0>0$ and satisfies the compatibility conditions. Then the system \eqref{linsys} admits a unique solution 
\begin{eqnarray*}
	u\in L^\infty\inner{[0,T_*]; X_{\rho_*, a}} 
\end{eqnarray*}
for some $\rho_*>0$,  $a>0$ and $T_*>0. $ 
\end{thm}

We now proceed the proof of the theorem \ref{th41}  through the following parabolic approximations. 
 
 \noindent{\bf The approximate solutions.}   
 Consider the following regularized system, for $\eps>0$, 
 \begin{equation}
	\label{linsys+}
		\left\{
	\begin{aligned}
		&\pare{\partial_t - \eps\partial_1^2-\partial_y^2 + \overline{\partial_3 u^{I,0}_3} y\partial_y } u^\eps  + \pare{  v^\eps+ \overline{u^{I,0}_1}} \partial_1 u^\eps\\
		&\qquad  \qquad + \pare{ u^\eps  -   u^\eps(t,x_1,0)} \partial_y u^\eps+ \overline{\partial_3 u^{I,0}_3} u^\eps+ \pare{\partial_1   u(t,x_1,0) + y\overline{\partial_1\partial_3 u^{I,0}_3}}  v= 0,\\
		&\partial_y u^\eps(t, x_1,0) = \overline{\partial_1u^{I,0}_1}(t,x_1),\quad
		\lim_{y\to +\infty} u^\eps(t,x_1,y) = 0, \\
		&u^\eps|_{t=0}= u_{0}(x_1,y).
	\end{aligned}
	\right.
\end{equation}

The above is a nonlinear    parabolic equation,  and  from classical theory    we  can   deduce the following local well-posedness result.
  
\begin{thm} \label{th42}
	Suppose the initial data $u_{0}\in X_{2\rho_0, a_0}$ for some $\rho_0>0$, $a_0>0$ and satisfies the compatibility conditions. Then the system \eqref{linsys+} admits a unique solution 
\begin{eqnarray*}
	u^\eps\in L^\infty\inner{[0,T_\eps]; X_{\rho_0, a}} 
\end{eqnarray*}
for some $0<a<a_0$ independent of $\eps$ and $ T_\eps>0$  depends  on $\eps$ . 
\end{thm}
 
\noindent{\bf Uniform estimates for the approximate solutions.}  We will  perform the uniform estimate with respect to $\eps$ for the approximate solutions $u^\eps$ given in the previous Theorem.  The main result here can be stated as follows. 

\begin{prop}\label{propapri}
Suppose $u^\eps\in L^\infty\inner{[0,T_\eps]; X_{\rho_0, a}}$	is a solution to the initial-boundary problem \eqref{linsys+}.  Then   there exists  $0<\rho_*\le \rho_0$, depending only on $\abs{u_0}_{X_{\rho_0,a_0}}$,   such that $u^\eps\in L^\infty\inner{[0,T_\eps]; X_{\rho_*, a}}$	 for all $\eps>0.$ Moreover
\begin{eqnarray}\label{uniest}
	 \norm{u^\eps}_{L^\infty\inner{[0,T_\eps]; X_{\rho_*,a} }} \leq  C \abs{u_0}_{X_{\rho_0,a_0} },
\end{eqnarray}
where $C$ is a constant   depending only on $a_0, \rho_0,  \tau, \norm{u^{I,0}_3}_{\mathcal A_\tau}$ and $\norm{u^{I,0}_1}_{\mathcal A_\tau} $,  but independent of  $\eps$.
\end{prop}

To prove the above proposition, we need   another two {\it auxiliary  norms}  $ \abs{\cdot}_{Y_{\rho, a}}$ and $ \abs{\cdot}_{Z_{\rho, a}}$ which are defined by
\begin{equation}\label{Yrho}
\begin{split}
 &\abs{u}_{Y_{\rho, a}}^2=\sum_{ m\leq 2 } \inner{\sum_{0\leq j\leq 1}  \norm{\comii{x_1}^\ell e^{a y^2}\partial_1^m\partial_y^j u}_{L^2(\mathbb R_+^2)}}^2\\
 &\qquad\qquad\qquad\qquad+\sum_{m\geq 3}    \inner{\sum_{0\leq j\leq 1} (m-1)^{1/2}  \rho^{-1/2}  \frac{\rho^{m-1}}{ (m-3)!}\norm{\comii{x_1}^\ell e^{a y^2}\partial_1^m\partial_y^j u}_{L^2(\mathbb R_+^2)}}^2,
\end{split}
\end{equation}
and
\begin{equation*}
\begin{split}
 \abs{u}_{Z_{\rho, a}}^2&=\sum_{ m\leq 2 } \inner{\sum_{1\leq j\leq 2}    \norm{ \comii{x_1}^\ell e^{a y^2}\partial_1^m\partial_y^j u}_{L^2(\mathbb R_+^2)} }^2\\
 &\qquad+\sum_{ m\geq 3 } \inner{\sum_{1\leq j\leq 2}   \frac{\rho^{m-1}}{ (m-3)!}\norm{ \comii{x_1}^\ell e^{a y^2}\partial_1^m\partial_y^j u}_{L^2(\mathbb R_+^2)} }^2.
 \end{split}
\end{equation*}

The following energy estimate is a key part to prove Proposition \ref{propapri}.
\begin{prop} \label{prenes}
 Let $u^\eps\in L^\infty\inner{[0,T_\eps]; X_{\rho_0, a}}$	be a solution to the initial-boundary problem \eqref{linsys+} and $0<\rho(t)\le \min\set{\rho_0/2, \tau/3}$ a smooth function. Then for any $ t\in[0, T_\eps], $ 
 \begin{equation}\label{53}
 	\begin{split}
 			&\abs{u^\eps(t)}_{X_{\rho(t),a }}^2
 		 +\int_0^{T_\eps} \abs{u^\eps(t)}_{Z_{\rho(t), a }}^2dt -\int_0^T \rho'(t) \abs{u^\eps(t)}_{Y_{\rho(t), a }}^2 dt\\
 		\leq &\abs{u_0}_{X_{\rho_0,a_0} }^2+ C \int_0^{T_\eps} \inner{\abs{\rho'(t)}\rho(t)^{-2}\abs{u^\eps(t)}_{X_{\rho(t), a} }+\abs{u^\eps(t)}_{X_{\rho(t), a} }^2+ \abs{u^\eps(t)}_{X_{\rho(t), a}}^4}dt\\
 		&+C \int_0^{T_\eps}  \abs{u^\eps(t)}_{Z_{\rho(t), a}} \abs{u^\eps(t)}_{Y_{\rho(t), a} }^2dt\, .
 	\end{split}
 \end{equation}
 \end{prop}
 
The proof of the proposition above is postponed to the next section, and we now use it to prove Proposition \ref{propapri}.

\begin{proof}[{\bf Proof of Proposition \ref{propapri}}]
	
  To simplify the notations we will use $C$ in the following discussion  to denote different suitable constants, which depend only on $a_0, \rho_0,  \tau, \norm{u^{I,0}_3}_{\mathcal A_\tau}$ and $\norm{u^{I,0}_1}_{\mathcal A_\tau} $,  but independent of  $\eps$.   
 
 Let  $\rho_\eps$ be   the solution to the differential equation:   
 \begin{equation}\label{eq+}
 \left\{
 \begin{split}
	\rho_\eps'(t)=-\abs{u^\eps(t)}_{Z_{\rho_\eps (t), a}}, \\
	\rho|_{t=0}=\min\set{\rho_0/2, \tau/3},
\end{split}
\right.
\end{equation}
or equivalently 
\begin{eqnarray}\label{des++}
	\rho_\eps(t)=\min\set{\rho_0/2, \tau/3}-\int_0^t \abs{u^\eps(s)}_{Z_{\rho_\eps(s), a}} ds.
\end{eqnarray}
Observe, for any $0<\rho,  \,\tilde\rho \leq \rho_0/2,$  we have
\begin{eqnarray*}
\abs{ \abs{u^\eps}_{Z_{\rho, a}}  -  \abs{u^\eps}_{Z_{\tilde\rho, a}}   }\leq C  \abs{u^\eps}_{Z_{\rho_0, a}} \abs{\rho-\tilde\rho},
\end{eqnarray*}
 which along with  Cauchy-Lipschitz  Theorem  gives the existence of $\rho_\eps$ to equation \eqref{eq+}. 
 Now  choosing  $\rho(t)=\rho_\eps(t)$ in \eqref{53}   and  observing \eqref{eq+}, 
we can rewrite  \eqref{53}  as
\begin{eqnarray*}
	&&\abs{u^\eps(t)}_{X_{\rho_\eps,a} }^2+\int_0^{T_\eps} \abs{u^\eps(t)}_{Z_{\rho_\eps,a} }^2 dt \\
	&\leq&\abs{u_0}_{X_{\rho_0,a_0} }^2+ C \int_0^{T_\eps} \inner{\abs{\rho_\eps'(t)}\rho_\eps^{-2}\abs{u^\eps}_{X_{\rho,a} }+\abs{u^\eps}_{X_{\rho_\eps,a} }^2+ \abs{u^\eps}_{X_{\rho_\eps,a}}^4}dt.
\end{eqnarray*}
Thus,  using \eqref{eq+}, 
\begin{equation}\label{Eq+++}
\begin{split}
	&\abs{u^\eps(t)}_{X_{\rho,a} }^2+\frac{1}{2}\int_0^{T_\eps} \abs{u^\eps(t)}_{Z_{\rho_\eps,a} }^2 dt \\
	\leq&  \abs{u_0}_{X_{\rho_0,a_0} }^2+ C \int_0^{T_\eps} \inner{ \rho_\eps^{-4}\abs{u^\eps}_{X_{\rho_\eps,a} }^2+\abs{u^\eps}_{X_{\rho_\eps,a} }^2 +\abs{u^\eps}_{X_{\rho_\eps,a}}^4}dt.
	\end{split}
\end{equation}
In view of \eqref{des++}  for $T_\eps$ be small sufficiently  we have 
\begin{eqnarray*}
	\forall~t\in[0,T_\eps], \quad \rho_\eps(t)\geq \frac{1}{8}\min\set{\rho_0,\tau/3}, 
\end{eqnarray*}
and thus  it follows from \eqref{Eq+++} that, for any $t\in[0,T_\eps], $
\begin{eqnarray*}
	 \abs{u^\eps(t)}_{X_{\rho,a} }^2+\frac{1}{2}\int_0^{T_\eps} \abs{u^\eps(t)}_{Z_{\rho_\eps,a} }^2 dt \leq  \abs{u_0}_{X_{\rho_0,a_0} }^2+ C \int_0^{T_\eps} \inner{  \abs{u^\eps}_{X_{\rho_\eps,a} }^2 +\abs{u^\eps}_{X_{\rho_\eps,a}}^4}dt,
\end{eqnarray*}
with $C$ depending only on $a_0, \rho_0,  \tau, \norm{u^{I,0}_3}_{\mathcal A_\tau}$ and $\norm{u^{I,0}_1}_{\mathcal A_\tau} $, but independent of $\eps$.  Thus  by general Gronwall  inequality, we conclude 
\begin{eqnarray}\label{ueps}
	 \abs{u^\eps(t)}_{X_{\rho_\eps,a} }^2\leq C  \abs{u_0}_{X_{\rho_0,a_0} }^2,
\end{eqnarray}
and  
\begin{eqnarray*}
	\int_0^{T_\eps} \abs{u^\eps(t)}_{Z_{\rho_\eps,a} }^2 dt \leq 3\abs{u_0}_{X_{\rho_0,a_0} }^2+\abs{u_0}_{X_{\rho_0,a_0} }^4.
\end{eqnarray*}
As a result, in view of \eqref{des++} we see
\begin{eqnarray*}
	\rho_\eps(t)&=&\min\set{\rho_0/2, \tau/3}-\int_0^t \abs{u^\eps(s)}_{Z_{\rho_\eps(s), a}} ds\\
	&\geq& \min\set{\rho_0/2, \tau/3}-t^{1/2}\inner{\int_0^{T_\eps} \abs{u^\eps(t)}_{Z_{\rho_\eps,a} }^2 dt}^{1/2}\\
	&\geq& \min\set{\rho_0/2, \tau/3}-t^{1/2}\inner{2\abs{u_0}_{X_{\rho_0,a_0} }^2+\abs{u_0}_{X_{\rho_0,a_0} }^4}^{1/2}.
\end{eqnarray*}
So  if we choose  $T_*$ such that 
\begin{eqnarray}\label{tstar}
T_*=	4^{-1}\inner{3\abs{u_0}_{X_{\rho_0,a_0} }^2+\abs{u_0}_{X_{\rho_0,a_0} }^4} ^{-1}\Big(\min\set{\rho_0/2, \tau/3}\Big)^2
\end{eqnarray}
Then      
\begin{eqnarray*}
	\forall~t\in [0,T_\eps] \subset [0,T_*], \quad \rho_\eps(t) \geq  \rho_*\stackrel{\rm def}{=} \frac{1}{4}\min\set{\rho_0, \tau/3}.
\end{eqnarray*}
By \eqref{ueps}, it follows that 
\begin{eqnarray*}
	\forall~t\in [0,T_\eps] \subset [0,T_*] ,\quad  \abs{u^\eps(t)}_{X_{\rho_*,a} }^2\leq  C\abs{u_0}_{X_{\rho_0,a_0} }^2.
\end{eqnarray*}
This completes the proof of Proposition \ref{propapri}.
\end{proof}

\begin{proof}[{\bf Completion of the proof of Theorem \ref{th41}}] Due to the uniform estimate \eqref{uniest}, we can extend the lifespan $T_\eps$ to $T_*$ with $T_*$ defined in \eqref{tstar}, following the standard bootstrap arguments.  Thus we see for any $\eps>0$ the system \eqref{linsys+} admits a unique solution $u^\eps\in   L^\infty\inner{[0,T_*]; X_{\rho_*, a}} $ such that
\begin{eqnarray*}
	\norm{u^\eps}_{L^\infty\inner{[0,T_\eps]; X_{\rho_*,a} }} \leq  C \abs{u_0}_{X_{\rho_0,a_0} },
\end{eqnarray*}
 with $T_*, \rho_*, a, C$ independent of $\eps.$  Thus letting $\eps\rightarrow 0,$ the compactness arguments show that the limit $u\in  L^\infty\inner{[0,T_*]; X_{\rho_*, a}}$   solves the system \eqref{linsys}, proving Theorem \ref{th41}.
\end{proof}

\noindent  
{\bf Proof of Theorem \ref{th:Prandtl} . } Taking 
$$
u=u^{B,1}_3,\quad v= - \int_{-\infty}^{x_1}\partial_y    u^{B,1}_3(t, z,y)dz,
$$ 
 the system \eqref{linsys} implies that the function 
$$
\Ucal^{p,1}_3 = u^{B,1}_3 + \overline{u^{I,1}_3} + y \overline{\partial_3 u^{I,0}_3}
$$
satisfies 
$$
    \left\{
    \begin{aligned}
    	&\partial_t \Ucal^{p,1}_3 - \partial_y^2 \Ucal^{p,1}_3 + \Ucal^{p,0}_1 \partial_1 \Ucal^{p,1}_3 + \Ucal^{p,1}_3 \partial_y \Ucal^{p,1}_3 + \overline{\partial_3 p^{I,1}} + y\overline{\partial_3^2 p^{I,0}} = 0,\\
    	 &\partial_y \Ucal^{p,1}_3(t,x_1,0) = 0,\\
    	&\Ucal^{p,1}_3(0,x_1,y) = u^{B,1}_{0,3} (x_1,y) + \overline{u^{I,1}_{0,3}}(x_1) + y\overline{\dd_3u^{I,0}_{0,3}}(x_1)\, ,
    \end{aligned}
	\right.
$$
with
$$
\Ucal^{p,0}_1= - \int_{-\infty}^{x_1}\partial_y    u^{B,1}_3(t, z,y)dz+ \overline{u^{I,0}_1}\,.
$$
So we   need to check  the boundary condition 
 \begin{equation}\label{bounadry3}
	\Ucal^{p,1}_3|_{y=0}= 	u^{B,1}_3(t, x_1, 0)+\overline{u^{I,1}_3}(t, x_1)=0.\,\,    
 	\end{equation}
For this purpose, we first use Theorem \ref{th41} to determine $u^{B,1}_3$, then use Theorem 	\ref{th:E2Dlin} to solve the linearized Euler system   \eqref{eq:Ordre0.5I} with the boundary condition 
$$
{u^{I ,1}_3}|_{x_3=0}=-u^{B,1}_3(t, x_1, 0)\,.   
$$
For the component $ \Ucal^{p,0}_1$, using the divergence-free properties of $u^{I, 0}$, we have firstly 
\begin{eqnarray*}
 \Ucal^{p, 0}_1|_{y=0}&=&- \int_{-\infty}^{x_1}\partial_y    u^{B,1}_3(t, z, 0)dz+ 
 \overline{u^{I,0}_1}(t, x_1)\\
 &=& \int_{-\infty}^{x_1}\overline{\partial_3u^{I, 0}_3}(t, z)dz+ \overline{u^{I,0}_1}(t, x_1)\\
 &=&- \int_{-\infty}^{x_1}\overline{\partial_1u^{I, 0}_1}(t, z)dz+ \overline{u^{I,0}_1}(t, x_1)\\
 &=&0.
\end{eqnarray*}
On the other hand, since $ u^{B,1}_3\in L^\infty\inner{[0,T_*]; X_{\rho_*, a}}$, we have  the   limit 
\begin{eqnarray*}
 \lim_{y\to +\infty}\Ucal^{p, 0}_1(t, x_1, y)&=&-  \lim_{y\to +\infty}\int_{-\infty}^{x_1}\partial_y    u^{B,1}_3(t, z, y)dz+ 
 \overline{u^{I,0}_1}(t, x_1)\\
 &=&\overline{u^{I,0}_1}(t, x_1).
\end{eqnarray*}
So the boundary conditions for $ \Ucal^{p,0}_1$ are satisfied.  
Finally, for the pressure term of the  first equation in \eqref{eq:PrRotE2}, once we obtain $\Ucal^{p, 1}_3$,   $\Ucal^{p, 0}_1$  and $\overline{\partial_1 p^{I,0}}$, it is enough to put
$$
 \partial_1 p^{B,0}=
-\partial_t \Ucal^{p,0}_1 + \partial_y^2 \Ucal^{p,0}_1 - \Ucal^{p,0}_1 \partial_1 \Ucal^{p,0}_1 - \Ucal^{p,1}_3 \partial_y \Ucal^{p,0}_1  - \overline{\partial_1 p^{I,0}} .
$$ 
We  then complete the proof of Theorem \ref{th:Prandtl}.

\section{Uniform energy estimates}
\label{section4}
In this  section we proceed  through the following lemmas to prove Proposition \ref{prenes}.      To simplify the notations  in the following proof we will write $u$ instead of $u^\eps,$  omitting the  superscript $\eps$,   and use $C$   in the following discussion  to denote different suitable constants, which depend only on $a_0, \rho_0,  \tau, \norm{u^{I,0}_3}_{\mathcal A_\tau}$ and $\norm{u^{I,0}_1}_{\mathcal A_\tau} $.

In view of the definition of $\abs{\cdot}_{X_{\rho,a}}$	it suffices to estimate terms 
\begin{eqnarray}\label{est1or}
	 \sum_{ m\leq 2 } \inner{\norm{\comii{x_1}^\ell e^{a y^2}\partial_1^m u}_{L^2(\mathbb R_+^2)} }+\sum_{ m\geq 3 } \inner{\frac{\rho^{m-1}}{ (m-3)!}\norm{\comii{x_1}^\ell e^{a y^2}\partial_1^m u}_{L^2(\mathbb R_+^2)}}
\end{eqnarray}
and
\begin{eqnarray}\label{est2or}
	 \sum_{ m\leq 2 } \inner{\norm{\comii{x_1}^\ell e^{a y^2}\partial_1^m\partial_y u}_{L^2(\mathbb R_+^2)} }+\sum_{ m\geq 3 } \inner{\frac{\rho^{m-1}}{ (m-3)!}\norm{\comii{x_1}^\ell e^{a y^2}\partial_1^m\partial_y u}_{L^2(\mathbb R_+^2)} }
\end{eqnarray}
Here we first treat the terms in  \eqref{est2or},  and the ones in  \eqref{est1or} can be deduced similarly with simpler arguments.   To do so,        we use the notation $\omega=\partial_y u$.  Then it follows from \eqref{linsys} that 
 \begin{equation}
	\label{linsysomega}
		\left\{
	\begin{aligned}
		&\pare{\partial_t - \partial_y^2 + \overline{\partial_3 u^{I,0}_3} y\partial_y} \omega  + \pare{  v+ \overline{u^{I,0}_1}} \partial_1 \omega\\
		&\qquad  \qquad + \pare{ u -   u(t,x_1,0)} \partial_y \omega+ 2\overline{\partial_3 u^{I,0}_3} \omega+ \pare{\partial_1   u(t,x_1,0) + y\overline{\partial_1\partial_3 u^{I,0}_3}}  \partial_yv\\
		&\qquad  \qquad +\inner{\partial_y v} \partial_1u+ \omega^2 +\overline{\partial_1\partial_3 u^{I,0}_3}v= 0,\\
		&\omega|_{y=0} = \overline{\partial_1u^{I,0}_1}(t,x_1),\\
				&\omega|_{t=0}= \partial_y u^{B,1}_{3, 0}.
	\end{aligned}
	\right.
\end{equation}
Thus  the function, defined by  
 \begin{eqnarray}\label{defver}
 \varphi_m =\comii{x_1}^\ell e^{a  y^2}\partial_1^m\omega (t)=\comii{x_1}^\ell e^{a  y^2}\partial_1^m\partial_y u (t),
 \end{eqnarray}
  solves the equation 
  \begin{equation*}
  	\begin{cases}
  		\pare{\partial_t - \partial_y^2 + \overline{\partial_3 u^{I,0}_3} y\partial_y} \varphi_m -a'(t)y^2 \varphi_m + \pare{v + \overline{u^{I,0}_1}} \partial_1 \varphi_m + \pare{u  - u (t,x_1,0)} \partial_y \varphi_m =\mathcal R^{m}(t),\\[10pt]
  	  \varphi_m \big|_{y=0}=\comii{x_1}^\ell\overline{\partial_1^{m+1} u^{I,0}_1}(t,x_1),\\[10pt]
  		\varphi_m \big|_{t=0}=\comii{x_1}^\ell e^{a  y^2}\partial_1^m \partial_y u^{B,1}_{3, 0},
  	\end{cases}
  \end{equation*}
  where  
\begin{eqnarray*}
\mathcal R^{m}(t)=\sum_{j=1}^{11}\mathcal R_j^{m}(t)
\end{eqnarray*}
with  
\begin{eqnarray*}
\mathcal R_1^{m}&=&-4a y\partial_y\varphi_m+4a^2 y^2\varphi_m -2a\varphi_m, \\
\mathcal R_2^{m}&=&2ay^2 \overline{\partial_3u^{I,0}_3} \varphi_m+2ay\inner{u-u(t,x_1,0)}\varphi_m,  \\
\mathcal R_3^{m}&=&\inner{\partial_1 \comii{x_1}^\ell }e^{a y^2}\inner{ v+\overline{u^{I,0}_1}}  \partial_1^{m} \omega, \\
\mathcal R_4^{m}&=&
- \sum_{k=1}^m {m\choose k} \comii{x_1}^\ell e^{a y^2}\inner{\partial_1^k \overline{\partial_3u^{I,0}_3}} y\partial_y\partial_1^{m-k}\omega,   \\
\mathcal R_5^{m}&=& - \sum_{k=1}^m {m\choose k} \comii{x_1}^\ell e^{a y^2}\inner{\partial_1^k v+\partial_1^k \overline{u^{I,0}_1}} \partial_1^{m-k+1}\omega, \\
\mathcal R_6^{m} &=& - \sum_{k=1}^m{m\choose k}  \comii{x_1}^\ell e^{a y^2}\inner{\partial_1^k u -\partial_1^k u (t,x_1,0)} \partial_1^{m-k} \partial_y\omega,\\
\mathcal R_7^{m} &=& - \sum_{k=0}^m {m\choose k} \comii{x_1}^\ell e^{a y^2}\inner{\partial_1^k  \overline{\partial_3u^{I,0}_3}} \partial_1^{m-k}\omega,\\
\mathcal R_8^{m}&=& - \sum_{k=0}^m{m\choose k} \comii{x_1}^\ell e^{a y^2} \inner{\partial_1^{k+1} u(t,x_1,0)+y \overline{\partial_1^{k+1} \partial_3u^{I,0}_3}} \partial_1^{m-k}\partial_yv,\\
\mathcal R_9^{m}&=& - \sum_{k=0}^m{m\choose k} \comii{x_1}^\ell e^{a y^2} \inner{\partial_1^{k+1} u} \partial_1^{m-k}\partial_yv,\\
\mathcal R_{10}^{m}&=& - \sum_{k=0}^m{m\choose k} \comii{x_1}^\ell e^{a y^2} \inner{\inner{\partial_1^{k} \omega} \partial_1^{m-k}\omega+ \overline{\partial_1^{k+1}\partial_3 u^{I,0}_3}\partial_1^{m-k}v}.
\end{eqnarray*}
From the first equation in \eqref{linsysomega},  it follows that
\begin{align}\label{maiequ}
	&  \inner{\pare{\partial_t - \partial_y^2 + \overline{\partial_3 u^{I,0}_3} y\partial_y} \varphi_m   , ~  \varphi_m }_{L^2(\mathbb R_+^2)} - \inner{a'(t)y^2 \varphi_m ,~\varphi_m }_{L^2(\mathbb R_+^2)}\nonumber  \\
	&\qquad\qquad+\inner{  \pare{v+ \overline{u^{I,0}_1}} \partial_1\varphi_m  + \pare{u  - u (t,x_1,0)} \partial_y  \varphi_m (t), ~  \varphi_m }_{L^2(\mathbb R_+^2)}  \\
	=& \inner{   \mathcal R^{m}(t), ~\varphi_m (t)}_{L^2(\mathbb R_+^2)},\notag
\end{align}
with $\mathcal R^{m}$     given above.

  In the following lemmas,  let $ 0<a(t)<a_0$   to be determined later, and let $0<\rho=\rho(t)\leq \min\set{\rho_0/2,\tau/3}$ be an arbitrary smooth function of $t.$  
 
\begin{lem} \label{R1+2}    A constants $C$ exists such that for any $N\geq 3$,  
	\begin{eqnarray*}
		 \sum_{m=3}^N \inner{ \frac{\rho^{m-1}}{(m-3)!}}^2  \inner{  \mathcal R_1^{m}, ~\varphi_m}_{L^2(\mathbb R_+^2)} 
	 \leq  C \sum_{m=3}^N \inner{ \frac{\rho^{m-1}}{(m-3)!}}^2  \norm{y\varphi_m}_{L^2(\mathbb R_+^2)}^2,
	\end{eqnarray*}
	and 
	\begin{eqnarray*}
		 \sum_{m=3}^N \inner{ \frac{\rho^{m-1}}{(m-3)!}}^2  \inner{   \mathcal R_2^{m}, ~\varphi_m}_{L^2(\mathbb R_+^2)}
	 \leq  C \sum_{m=3}^N \inner{ \frac{\rho^{m-1}}{(m-3)!}}^2  \norm{y\varphi_m}_{L^2(\mathbb R_+^2)}^2+ C  \abs{u}_{X_{\rho,a}}^4.
	\end{eqnarray*}
\end{lem}
\begin{proof}
	We have, integrating by parts, 
	\begin{eqnarray*}
		 \inner{   \mathcal R_1^{m}, ~\varphi_m}_{L^2(\mathbb R_+^2)}=
		 \inner{ 4a^2y^2\varphi_m, ~\varphi_m}_{L^2(\mathbb R_+^2)}=  4a^2 \norm{y\varphi_m}_{L^2(\mathbb R_+^2)}^2
	\end{eqnarray*}
	Direct verification shows
	\begin{eqnarray*}
	 \inner{   \mathcal R_2^{m}, ~\varphi_m}_{L^2(\mathbb R_+^2)}\leq
	\inner{2a \norm{\overline{\partial_3u^{I,0}_3}}_{L^\infty}+ a^2} \norm{y\varphi_m}_{L^2(\mathbb R_+^2)}^2 +  4\norm{u}_{L^\infty(\mathbb R_+^2)}^2\norm{\varphi_m}_{L^2(\mathbb R_+^2)}^2.
	\end{eqnarray*}
	Observe 
	\begin{eqnarray*}
		 \norm{\overline{\partial_3u^{I,0}_3}}_{L^\infty}\leq C \norm{u^{I,0}_3}_{G_\tau}
	\end{eqnarray*}
	and
	\begin{eqnarray*}
		 \sum_{m=3}^N \inner{ \frac{\rho^{m-1}}{(m-3)!}}^2\norm{\varphi_m}_{L^2(\mathbb R_+^2)}^2\leq  \sum_{m=3}^\infty \inner{ \frac{\rho^{m-1}}{(m-3)!}}^2\norm{\varphi_m}_{L^2(\mathbb R_+^2)}^2\leq \abs{u}_{X_{\rho,a}},
	\end{eqnarray*}
	and thus  the desired results follow, completing the proof.  
\end{proof}

\begin{lem}\label{R8}
There exists a constant $C$ such that for  any   $ \rho$ with $0<\rho\leq  \tau/3$,   we have
\begin{eqnarray*}
	 \sum_{m=3}^N \inner{ \frac{\rho^{m-1}}{(m-3)!}}^2  \inner{   \mathcal R_8^{m}, ~\varphi_m}_{L^2(\mathbb R_+^2)}
	 \leq  {1\over 8}     \abs{u}_{Z_{\rho,a}}^2 +C \sum_{m=3}^N \inner{ \frac{\rho^{m-1}}{(m-3)!}}^2 \norm{y\varphi_m}_{L^2(\mathbb R_+^2)}^2+C  \abs{u}_{Z_{\rho,a}} \abs{u}_{Y_{\rho,a} }^2.   
\end{eqnarray*}
 
\end{lem}

\begin{proof}   
Recall   $\mathcal R_8^{m,j}$ can be written as, for any $\tilde\eps>0,$
\begin{eqnarray*}
	  \inner{   \mathcal R_8^{m}, ~\varphi_m}_{L^2(\mathbb R_+^2)}
	&= & \inner{-\sum_{k=0}^m{m\choose k} \comii{x_1}^\ell e^{a y^2}   \inner{\partial_1^{k+1} u(x_1,0)}  \partial_1^{m-k}\partial_y v,  ~ \varphi_m}_{L^2(\mathbb R_+^2)} \\
	&&+  \inner{- \sum_{k=0}^m{m\choose k}  \comii{x_1}^\ell e^{a y^2} \inner{ y  \overline{\partial_1^{k+1}\partial_3u^{I,0}_3} }\partial_1^{m-k}\partial_y v,  ~ \varphi_m}_{L^2(\mathbb R_+^2)}\\
	&\leq &\sum_{k=0}^m{m\choose k}    \norm{\comii{x_1}^\ell  \partial_1^{k+1} u}_{L^\infty_y\inner{\mathbb R_+;~L_{x_1}^2(\mathbb R)}}\norm{ e^{a y^2} \partial_1^{m-k}\partial_y v}_{L^2_y\inner{\mathbb R_+;~L_{x_1}^\infty(\mathbb R)}} \norm{\varphi_m}_{L^2(\mathbb R_+^2)} \\
	&&+\tilde\eps\bigg[\sum_{k=0}^m{m\choose k}    \norm{\comii{x_1}^\ell  \overline{\partial_1^{k+1}\partial_3u^{I,0}_3} }_{ L_{x_1}^2(\mathbb R)}\norm{ e^{a y^2} \partial_1^{m-k}\partial_y v}_{L^2_y\inner{\mathbb R_+;~L_{x_1}^\infty(\mathbb R)}}\bigg]^2 \\
	&&+C_{\tilde\eps}  \norm{y\varphi_m}_{L^2(\mathbb R_+^2)}^2.
\end{eqnarray*}
Then it suffices to show that
\begin{equation} \label{cru1}
    \begin{split}
   &  \sum_{m=3}^N\inner{  \frac{\rho^{m-1}}{(m-3)!}}^2 \sum_{k=0}^m{m\choose k}    \norm{\comii{x_1}^\ell  \partial_1^{k+1} u}_{L^\infty_y\inner{\mathbb R_+;~L_{x_1}^2(\mathbb R)}}\norm{ e^{a y^2} \partial_1^{m-k}\partial_y v}_{L^2_y\inner{\mathbb R_+;~L_{x_1}^\infty(\mathbb R)}} \norm{\varphi_m}_{L^2(\mathbb R_+^2)}     \\
    \leq &  C  \abs{u}_{Z_{\rho,a}} \abs{u}_{Y_{\rho,a} }^2\, ,
\end{split}
\end{equation}
and
\begin{equation}\label{cru2}
    \begin{split}
   &  \sum_{m=3}^N \inner{  \frac{\rho^{m-1}}{(m-3)!}}^2  \com{ \sum_{k=0}^m{m\choose k}    \norm{\comii{x_1}^\ell  \overline{\partial_1^{k+1}\partial_3u^{I,0}_3} }_{ L_{x_1}^2(\mathbb R)}\norm{ e^{a y^2} \partial_1^{m-k}\partial_y v}_{L^2_y\inner{\mathbb R_+;~L_{x_1}^\infty(\mathbb R)}}}^2\\
    \leq & C     \abs{u}_{Z_{\rho,a}} ^2\, .
\end{split}
\end{equation}
We will proceed to prove the above estimate through the following steps. 

\noindent
{\bf  Step (a)} We begin with  several estimates to be used later in the proof.   Firstly in view of the definition of $\abs{\cdot}_{Y_{\rho,a}}$ given in \eqref{Yrho},  we may write
\begin{eqnarray*}
	\abs{u}_{Y_{\rho,a}}^2=\sum_{m=0}^{+\infty}\abs{u}_{Y_{\rho,a,m}}^2
\end{eqnarray*} 
where   $\abs{u}_{Y_{\rho,a,m}} $ is defined by 
\begin{eqnarray*}
	 \abs{u}_{Y_{\rho,a,m}} =
	 \begin{cases}
	 \sum_{0\leq j\leq 1}  \norm{\comii{x_1}^\ell e^{a y^2}\partial_1^m\partial_y^j u}_{L^2(\mathbb R_+^2)},\quad 0\leq m\leq 2\\[8pt]
	  \sum_{0\leq j\leq 1} (m-1)^{1/2}  \rho^{-/2}  \frac{\rho^{m-1}}{ (m-3)!} \norm{\comii{x_1}^\ell e^{a y^2}\partial_1^m\partial_y^j u}_{L^2(\mathbb R_+^2)},\quad m\geq 3.
	 \end{cases}
\end{eqnarray*}
Thus
\begin{eqnarray}\label{ta}
\norm{\varphi_m}_{L^2(\mathbb R_+^2)}\leq 
\begin{cases}
      \abs{u}_{Y_{\rho,a}},\quad 0\leq m\leq 2,\\[8pt]
		  \abs{u}_{Y_{\rho,a,m}} m^{-1/2}\rho^{1/2}\frac{(m-3)!}{\rho^{m-1}},\quad m\geq 3,
\end{cases}
\end{eqnarray}

Next from the relations \eqref{relauv} it follows  that  
\begin{eqnarray*}
\norm{ e^{a y^2}   \partial_y v}_{L^2_y\inner{\mathbb R_+;~L_{x_1}^\infty(\mathbb R)}} \leq  C \norm{\comii{x_1}^\ell e^{a y^2} \partial_y^2 u}_{L^2\inner{\mathbb R_+^2}}  \leq C\abs{u}_{Z_{\rho,a}},
\end{eqnarray*}
and that for $j\geq 1,$
\begin{eqnarray*}
 \norm{ e^{a y^2}\partial_1^j \partial_yv}_{L^2_y\inner{\mathbb R_+;~L_{x_1}^\infty(\mathbb R)}}&=&\norm{ e^{a y^2}  \partial_1^{j-1}\partial_y^2 u }_{L^2_y\inner{\mathbb R_+;~L_{x_1}^\infty(\mathbb R)}}\\
&\leq& C \norm{\comii{x_1}^\ell e^{a y^2}  \partial_1^{j}\partial_y^2 u}_{L^2\inner{\mathbb R_+^2}}    \\
&\leq &\left\{
\begin{array}{lll}
 C \abs{u}_{Z_{\rho,a}},\quad && 1\leq j\leq 2,\\[6pt]
  C \abs{u}_{Z_{\rho,a,j}} \frac{(j-3)!}{\rho^{j-1}},\quad &&  j\geq 3,
\end{array}
\right.
\end{eqnarray*}
where   $\abs{u}_{Z_{\rho,a, k}} $ is defined by  the relation $\abs{u}_{Z_{\rho,a}}=\sum_{k\geq 0}\abs{u}_{Z_{\rho,a, k}}^2$, so that
\begin{eqnarray*}
	 \abs{u}_{Z_{\rho,a,k}} =
	 \begin{cases}
	 \sum_{1\leq j\leq 2}  \norm{\comii{x_1}^\ell e^{a y^2}\partial_1^k\partial_y^j u}_{L^2(\mathbb R_+^2)},\quad 0\leq k\leq 2\\[8pt]
	  \sum_{1\leq j\leq 2}   \frac{\rho^{k-1}}{ (k-3)!} \norm{\comii{x_1}^\ell e^{a y^2}\partial_1^k\partial_y^j u}_{L^2(\mathbb R_+^2)},\quad k\geq 3.
	 \end{cases}
\end{eqnarray*}
Thus we conclude 
\begin{eqnarray}\label{tc+}
\norm{ e^{a y^2}\partial_1^j \partial_y v}_{L^2_y\inner{\mathbb R_+;~L_{x_1}^\infty(\mathbb R)}} \leq \left\{
\begin{array}{lll}
C \abs{u}_{Z_{\rho,a}},\quad && 0\leq j \leq 2,\\[6pt]
C \abs{u}_{Z_{\rho,a,j}} \frac{(j-3)!}{\rho^{j-1}},\quad &&  j\geq 3.
\end{array}
\right.
\end{eqnarray}
Using the Sobolev  inequality 
\begin{eqnarray*}
  \norm{\comii{x_1}^\ell \partial_1^{j} u}_{L^\infty_y\inner{\mathbb R_+;~L_{x_1}^2(\mathbb R)}}    \leq C \norm{\comii{x_1}^\ell   \partial_1^{j}  u}_{L^2(\mathbb R_+)}+C \norm{\comii{x_1}^\ell   \partial_1^{j} \partial_y u}_{L^2(\mathbb R_+)}, 
\end{eqnarray*}
gives
\begin{eqnarray}\label{tb} 
\norm{\comii{x_1}^\ell \partial_1^{j} u }_{L^\infty_y\inner{\mathbb R_+;~L_{x_1}^2(\mathbb R)}}    \leq   \left \{ \begin{array}{lll}
C\abs{u}_{Y_{\rho,a}},\qquad{\rm if}~0\leq j\leq 2,\\[6pt]
C\abs{u}_{Y_{\rho,a,j} } j^{-1/2}\rho^{1/2}\frac{(j-3)!}{\rho^{j-1}},\qquad{\rm if}~j\geq 3.
\end{array}\right.
\end{eqnarray} 
Finally, 
\begin{eqnarray}\label{eseu}
	\forall~k\geq 0,\quad   \norm{\comii{x_1}^\ell  \overline{\partial_1^{k+1}\partial_3u^{I,0}_3} }_{ L_{x_1}^2(\mathbb R)}\leq  C \norm{u^{I,0}_3}_{\mathcal A_\tau} \frac{(k+3)!}{\tau^{k+3}}
\end{eqnarray}
due to  \eqref{eseuler}. 

\noindent
{\bf  Step (b)}. We now prove \eqref{cru2}. For this purpose we use \eqref{eseu} and \eqref{tc+} to calculate
\begin{eqnarray*}
	&& \sum_{k=0}^m{m\choose k}    \norm{\comii{x_1}^\ell  \overline{\partial_1^{k+1}\partial_3u^{I,0}_3} }_{ L_{x_1}^2(\mathbb R)}\norm{ e^{a y^2} \partial_1^{m-k}\partial_y v}_{L^2_y\inner{\mathbb R_+;~L_{x_1}^\infty(\mathbb R)}}\\
	&\leq & C \norm{u^{I,0}_3}_{\mathcal A_\tau}   \sum_{k=0}^{m-3}   \frac{m!}{k!(m-k)! }\frac{(k+3)!}{\tau^{k+3}}\frac{(m-k-3)!}{\rho^{m-k-1}} \abs{u}_{Z_{\rho,a, m-k}}\\
&&+	C \norm{u^{I,0}_3}_{\mathcal A_\tau}    \sum_{k=m-2}^{m}   \frac{m!}{k!(m-k)! }\frac{(k+3)!}{\tau^{k+3}} \abs{u}_{Z_{\rho,a, m-k}} \\
&\leq& C \frac{
	 (m-3)!}{\rho^{m-1}}  \sum_{k=0}^{m-3}   \frac{m^3}{k^3(m-k-2)^3}\frac{2^k\rho^k}{\tau^{k+3}}\abs{u}_{Z_{\rho,a, m-k}}+ C \frac{
	 (m-3)!}{\rho^{m-1}}\sum_{k=m-2}^{m}  \frac{2^k\rho^{m-1}}{\tau^{k+3}} \abs{u}_{Z_{\rho,a, m-k}}\\
	 &\leq& C \tau^{-3}\frac{
	 (m-3)!}{\rho^{m-1}}  \sum_{k=0}^{m-3}    \frac{2^k\rho^k}{\tau^{k}}\abs{u}_{Z_{\rho,a, m-k}}+ C\tau^{-3} \frac{
	 (m-3)!}{\rho^{m-1}}\sum_{k=m-2}^{m}  \frac{2^k\rho^{m-1}}{\tau^{k}} \abs{u}_{Z_{\rho,a, m-k}},
\end{eqnarray*}
which yields
\begin{eqnarray*}
	&&  \sum_{m=3}^N \inner{  \frac{\rho^{m-1}}{(m-3)!}}^2  \com{ \sum_{k=0}^m{m\choose k}    \norm{\comii{x_1}^\ell  \overline{\partial_1^{k+1}\partial_3u^{I,0}_3} }_{ L_{x_1}^2(\mathbb R)}\norm{ e^{a y^2} \partial_1^{m-k}\partial_y v}_{L^2_y\inner{\mathbb R_+;~L_{x_1}^\infty(\mathbb R)}}}^2\\
	&\leq& C\sum_{m=3}^N   \inner{ \sum_{k=0}^{m-3}   \frac{2^k\rho^k}{\tau^{k}}\abs{u}_{Z_{\rho,a, m-k}}}^2+ C\sum_{m=3}^N   \inner{\sum_{k=m-2}^{m}  \frac{2^k\rho^{m-1}}{\tau^{k}} \abs{u}_{Z_{\rho,a, m-k}}}^2  
\end{eqnarray*}
On the other hand,   by virtue of  Young's inequality for discrete convolution (cf. \cite[Theorem 20.18]{HR} )  we have 
\begin{eqnarray*}
   \sum_{m=3}^N   \inner{ \sum_{k=0}^{m-3}   \frac{2^k\rho^k}{\tau^{k}}\abs{u}_{Z_{\rho,a, m-k}}}^2 
	  \leq C   \inner{\sum_{k=0}^{N}   \frac{2^k\rho^k}{\tau^{k}}  }^2 \sum_{k=0}^{N}\abs{u}_{Z_{\rho,a, k}}^2  \leq  C   \abs{u}_{Z_{\rho,a}}^2,\end{eqnarray*}
since $\rho\leq \tau/3.$  And direct computation yields 
\begin{eqnarray*}
	 \sum_{m=3}^N   \inner{\sum_{k=m-2}^{m}  \frac{2^k\rho^{m-1}}{\tau^{k}} \abs{u}_{Z_{\rho,a, m-k}}}\leq  C   \abs{u}_{Z_{\rho,a}}^2.
	\end{eqnarray*}
Then we obtain \eqref{cru2}, combining the above inequalities. 

\noindent
{\bf  Step (c)}.  Now we check \eqref{cru1} and write   
\begin{eqnarray*}
&&   \sum_{k=0}^m{m\choose k}    \norm{\comii{x_1}^\ell  \partial_1^{k+1} u}_{L^\infty_y\inner{\mathbb R_+;~L_{x_1}^2(\mathbb R)}}\norm{ e^{a y^2} \partial_1^{m-k}\partial_y v}_{L^2_y\inner{\mathbb R_+;~L_{x_1}^\infty(\mathbb R)}} \norm{\varphi_m}_{L^2(\mathbb R_+^2)} \\
&\leq & S_1+S_2+S_3
\end{eqnarray*}
with 
\begin{eqnarray*}
S_1&=&   \sum_{k=0}^2{m\choose k}    \norm{\comii{x_1}^\ell  \partial_1^{k+1} u}_{L^\infty_y\inner{\mathbb R_+;~L_{x_1}^2(\mathbb R)}}\norm{ e^{a y^2} \partial_1^{m-k}\partial_y v}_{L^2_y\inner{\mathbb R_+;~L_{x_1}^\infty(\mathbb R)}} \norm{\varphi_m}_{L^2(\mathbb R_+^2)},\\
S_2 &=& \sum_{k=3}^{m-3}      {m\choose k}    \norm{\comii{x_1}^\ell  \partial_1^{k+1} u}_{L^\infty_y\inner{\mathbb R_+;~L_{x_1}^2(\mathbb R)}}\norm{ e^{a y^2} \partial_1^{m-k}\partial_y v}_{L^2_y\inner{\mathbb R_+;~L_{x_1}^\infty(\mathbb R)}} \norm{\varphi_m}_{L^2(\mathbb R_+^2)}
\end{eqnarray*}
and
\begin{eqnarray*}
S_3 &=& \sum_{k=m-2}^{m}   {m\choose k}    \norm{\comii{x_1}^\ell  \partial_1^{k+1} u}_{L^\infty_y\inner{\mathbb R_+;~L_{x_1}^2(\mathbb R)}}\norm{ e^{a y^2} \partial_1^{m-k}\partial_y v}_{L^2_y\inner{\mathbb R_+;~L_{x_1}^\infty(\mathbb R)}} \norm{\varphi_m}_{L^2(\mathbb R_+^2)}.
\end{eqnarray*}
For the term $S_{2,m}$,  we use    \reff{ta}, \reff{tc+} and \reff{tb} to compute
\begin{eqnarray*}
	S_{2,m} &=&\inner{ \frac{\rho^{m-1}}{(m-3)!}}^2\,\, \sum_{k=3}^{m-3}      {m\choose k}    \norm{\comii{x_1}^\ell  \partial_1^{k+1} u}_{L^\infty_y\inner{\mathbb R_+;~L_{x_1}^2(\mathbb R)}}\norm{ e^{a y^2} \partial_1^{m-k}\partial_y v}_{L^2_y\inner{\mathbb R_+;~L_{x_1}^\infty(\mathbb R)}} \norm{\varphi_m}_{L^2(\mathbb R_+^2)}\\
	 &\leq &C   \inner{ \frac{\rho^{m-1}}{(m-3)!}}^2\,\,\sum_{k=3}^{m-3}\frac{m!}{k!(m-k)! }\left[ k^{-1/2}\rho^{1/2}\frac{(k-2)!}{\rho^{k}} \abs{u}_{Y_{\rho,a,k+1} } \right] \frac{(m-k-3)!}{\rho^{m-k-1}}\abs{u}_{Z_{\rho,a,m-k}}\\
	 &&\qquad\qquad\qquad \times   m^{-1/2}\rho^{1/2} \frac{(m-3)!}{\rho^{m-1}}\abs{u}_{Y_{\rho,a,m} } \\
	&\leq & C  \rho \abs{u}_{Y_{\rho,a,m} }  \sum_{k=3}^{m-3}\frac{m^3}{k^2(m-k-2)^3  }k^{-1/2}m^{-1/2} \abs{u}_{Y_{\rho,a,k+1} }\abs{u}_{Z_{\rho,a,m-k}}\\
	&\leq & C   \rho \abs{u}_{Y_{\rho,a,m} }  \inner{\sum_{k=3}^{m-3}\frac{m^3}{k^2(m-k-2)^3  }k^{-1/2}m^{-1/2} \abs{u}_{Y_{\rho,a,k+1} }^2}^{1/2} \\
	&&\qquad\qquad\qquad \times \inner{\sum_{k=3}^{m-3}\frac{m^3}{k^2(m-k-2)^3  }k^{-1/2}m^{-1/2}\abs{u}_{Z_{\rho,a,m-k}}^2}^{1/2}\\
		&\leq & C \rho \abs{u}_{Z_{\rho,a,m}} \abs{u}_{Y_{\rho,a,m} }^2,
\end{eqnarray*}
and thus
\begin{eqnarray*}
	\sum_{m=3}^N S_{2,m}& \leq& C\rho \inner{\sum_{m=3}^N \abs{u}_{Y_{\rho,a,m} }^2}^{1/2}\inner{\sum_{m=3}^N \com{\sum_{k=3}^{m-3}\frac{m^3}{k^2(m-k-2)^3  }k^{-1/2}m^{-1/2} \abs{u}_{Y_{\rho,a,k+1} }\abs{u}_{Z_{\rho,a,m-k}}}^2 }^{1/2}\\
	& \leq& C\rho \abs{u}_{Y_{\rho,a}} \inner{\sum_{m=3}^N \com{\sum_{k=3}^{m-3}\frac{1}{k^2  }  \abs{u}_{Y_{\rho,a,k+1} }\abs{u}_{Z_{\rho,a,m-k}}}^2 }^{1/2}\\
	&&\qquad\qquad \qquad + C\rho \abs{u}_{Y_{\rho,a}} \inner{\sum_{m=3}^N \com{\sum_{k=3}^{m-3}\frac{1}{(m-k-2)^3  }  \abs{u}_{Y_{\rho,a,k+1} }\abs{u}_{Z_{\rho,a,m-k}}}^2 }^{1/2} 
\end{eqnarray*}
the last inequality following from the fact that 
\begin{eqnarray*}
\forall~ 3\leq k\leq m-3,\quad \frac{m^3}{k^2(m-k-2)^3  }k^{-1/2}m^{-1/2}\leq C \inner{\frac{1}{k^2}+ \frac{1}{(m-k-2)^3}}. 
\end{eqnarray*}
Moreover,   by virtue of  Young's inequality for discrete convolution (cf. \cite[Theorem 20.18]{HR} ) 
we obtain 
\begin{eqnarray*}
	\inner{\sum_{m=3}^N \com{\sum_{k=3}^{m-3}\frac{1}{k^2  }  \abs{u}_{Y_{\rho,a,k+1} }\abs{u}_{Z_{\rho,a,m-k}}}^2 }^{1/2}&\leq &C\inner{\sum_{m=3}^{+\infty}  \abs{u}_{Z_{\rho,a,m}}^2 }^{1/2}\sum_{k=3}^{+\infty} \frac{1}{k^2} \abs{u}_{Y_{\rho,a,k}}\\
	&\leq &C  \abs{u}_{Z_{\rho,a}} \inner{ \sum_{k=1}^{+\infty}  \abs{u}_{Y_{\rho,a,k}}^2}^{1/2}\inner{ \sum_{k=1}^{+\infty} \frac{1}{k^4}}^{1/2}\\
	&\leq &C  \abs{u}_{Z_{\rho,a}}\abs{u}_{Y_{\rho,a}}.
\end{eqnarray*}
Similarly 
\begin{eqnarray*}
	\inner{\sum_{m=3}^N \com{\sum_{k=3}^{m-3}\frac{1}{(m-k-2)^3  }  \abs{u}_{Y_{\rho,a,k+1} }\abs{u}_{Z_{\rho,a,m-k}}}^2 }^{1/2}
	 \leq C  \abs{u}_{Z_{\rho,a}}\abs{u}_{Y_{\rho,a}}.
\end{eqnarray*}
Combining these inequality we conclude
\begin{eqnarray*}
	\sum_{m=3}^N S_{2,m}\leq C  \rho \abs{u}_{Z_{\rho,a}}\abs{u}_{Y_{\rho,a}}^2\leq C    \abs{u}_{Z_{\rho,a}}\abs{u}_{Y_{\rho,a}}^2.
\end{eqnarray*}
The estimates on the rest two terms $S_1$ and $S_3$ can be deduced similarly and directly, and we have
\begin{eqnarray*}
	\sum_{m=3}^N \inner{S_{1,m}+S_{3,m}}  
	 \leq  C  \abs{u}_{Z_{\rho,a}} \abs{u}_{Y_{\rho,a} }^2,
\end{eqnarray*} 
proving \eqref{cru1}.  The proof of Lemma \ref{R8} is complete. 
\end{proof}

\begin{lem}\label{Rothers}
	A constant $C$ exists such that  
	\begin{eqnarray*}
		&&   \sum_{m=3}^N \inner{ \frac{\rho^{m-1}}{(m-3)!}}^2  \inner{   \mathcal R_3^{m}, ~\varphi_m}_{L^2(\mathbb R_+^2)} \leq C \inner{\abs{u}_{X_{\rho,a,m}}^3 +  \abs{u}_{X_{\rho,a,m} }^2},\\
		&&   \sum_{m=3}^N \inner{ \frac{\rho^{m-1}}{(m-3)!}}^2  \inner{   \mathcal R_4^{m}, ~\varphi_m}_{L^2(\mathbb R_+^2)} \leq {1\over8}    \abs{u}_{Z_{\rho,a}}^2 +C\sum_{m=3}^N \inner{ \frac{\rho^{m-1}}{(m-3)!}}^2 \norm{y\varphi_m}_{L^2(\mathbb R_+^2)}^2,\\
		&&   \sum_{m=3}^N \inner{ \frac{\rho^{m-1}}{(m-3)!}}^2  \inner{   \mathcal R_5^{m}, ~\varphi_m}_{L^2(\mathbb R_+^2)} \leq C \inner{\abs{u}_{X_{\rho,a,m}}^3 +  \abs{u}_{X_{\rho,a,m} }^2},\\
		&&   \sum_{m=3}^N \inner{ \frac{\rho^{m-1}}{(m-3)!}}^2  \inner{   \mathcal R_6^{m}, ~\varphi_m}_{L^2(\mathbb R_+^2)} \leq C    \abs{u}_{Z_{\rho,a}} \abs{u}_{Y_{\rho,a} }^2,\\
		&&   \sum_{m=3}^N \inner{ \frac{\rho^{m-1}}{(m-3)!}}^2  \inner{   \mathcal R_7^{m}, ~\varphi_m}_{L^2(\mathbb R_+^2)} \leq C     \abs{u}_{X_{\rho,a} }^2,\\
		&&   \sum_{m=3}^N \inner{ \frac{\rho^{m-1}}{(m-3)!}}^2  \inner{   \mathcal R_9^{m}, ~\varphi_m}_{L^2(\mathbb R_+^2)} \leq C    \abs{u}_{Z_{\rho,a}} \abs{u}_{Y_{\rho,a} }^2,\\
		&& \sum_{m=3}^N \inner{ \frac{\rho^{m-1}}{(m-3)!}}^2  \inner{   \mathcal R_{10}^{m}, ~\varphi_m}_{L^2(\mathbb R_+^2)} \leq C \inner{\abs{u}_{X_{\rho,a}}^3 +  \abs{u}_{X_{\rho,a} }^2}.
	\end{eqnarray*}
\end{lem}

\begin{proof}
	The treatment of $\mathcal R_6, \mathcal R_9$ is exactly the same as in the proof of  \eqref{cru1}.  The other terms can be deduced similarly by following the proof in Lemma \ref{R8} with slightly changes,  and the arguments here will be simpler since there is no the highest derivative $\partial_1^{m+1}$ involved.  This means we can perform the estimates with the norm $Y_{\rho,a}$ in Lemma \ref{R8} replaced by $X_{\rho,a}$ here.  So we omit   the proof for brevity.
\end{proof}

Combining the estimates in Lemma \ref{R1+2}-Lemma \ref{Rothers}, we have
\begin{cor}\label{cor0222}
There are two constants $C, C_0$ such that
	\begin{eqnarray*}
		  &&\sum_{m=3}^N \inner{ \frac{\rho^{m-1}}{(m-3)!}}^2   \inner{\sum_{k=1}^{10}\inner{   \mathcal R_k^{m}, ~\varphi_m}_{L^2(\mathbb R_+^2)}}\\
		  & \leq&   {1\over 4}     \abs{u}_{Z_{\rho,a}}^2 +C_0 \sum_{m=3}^N \inner{ \frac{\rho^{m-1}}{(m-3)!}}^2 \norm{y\varphi_m}_{L^2(\mathbb R_+^2)}^2+C  \abs{u}_{Z_{\rho,a}} \abs{u}_{Y_{\rho,a} }^2+C \inner{\abs{u}_{X_{\rho,a}}^2 +  \abs{u}_{X_{\rho,a} }^4}.
	\end{eqnarray*}
\end{cor}

\begin{lem}\label{lowleft}
We have
\begin{eqnarray*}
&&\inner{\pare{\partial_t - \partial_y^2 + \overline{\partial_3 u^{I,0}_3} y\partial_y} \varphi_m   , ~  \varphi_m }_{L^2(\mathbb R_+^2)} - \inner{a'(t)y^2 \varphi_m ,~\varphi_m }_{L^2(\mathbb R_+^2)}\nonumber  \\
	&&\qquad\qquad+\inner{  \pare{v+ \overline{u^{I,0}_1}} \partial_1\varphi_m  + \pare{u  - u (t,x_1,0)} \partial_y  \varphi_m (t), ~  \varphi_m }_{L^2(\mathbb R_+^2)} \\
&\geq&	\frac{1}{2}\frac{d}{dt} \norm{\varphi_m}_{L^2(\mathbb R_+^2)} ^2+\norm{\partial_y\varphi_m}_{L^2(\mathbb R_+^2)} ^2-a'(t) \norm{y  \varphi_m}_{L^2(\mathbb R_+^2)}^2
	 \\
	&&+\frac{d}{dt}\int_{\mathbb R_{x_1}} \comii{x_1}^\ell\overline{\partial_1^{m+1}u^{I,0}_1}(t,x_1) \comii{x_1}^\ell\partial_1^mu(t,x_1, 0)\,dx_1-C \inner{ \frac{(m-3)!}{\rho^{m-1}}}^2\abs{u}_{X_{\rho,a,m} }^2.
\end{eqnarray*}	
\end{lem}

\begin{proof}
Firstly we calculate, integrating by parts and using the relation \eqref{relauv},
\begin{eqnarray}\label{seesf}
&&	\abs{\inner{  \overline{\partial_3 u^{I,0}_3} y\partial_y\varphi_m  , ~  \varphi_m }_{L^2(\mathbb R_+^2)} }+\abs{\inner{  \pare{v+ \overline{u^{I,0}_1}} \partial_1\varphi_m  + \pare{u  - u (t,x_1,0)} \partial_y  \varphi_m, ~  \varphi_m}_{L^2(\mathbb R_+^2)}}\nonumber \\
	&\leq &\frac{1}{2} \inner{\norm{\overline{\partial_3 u^{I,0}_3} }_{L^\infty}+\norm{\overline{\partial_3 u^{I,0}_1} }_{L^\infty}}\norm{ \varphi_m}_{L^2(\mathbb R_+^2)} ^2.
\end{eqnarray}
Integrating by parts and  using the boundary condition in \eqref{linsysomega},  we have
\begin{eqnarray}\label{fiest}
  \inner{\pare{\partial_t - \partial_y^2 } \varphi_m  , ~  \varphi_m }_{L^2(\mathbb R_+^2)} 
	 &=&\frac{1}{2}\frac{d}{dt} \norm{\varphi_m}_{L^2(\mathbb R_+^2)} ^2+\norm{\partial_y\varphi_m}_{L^2(\mathbb R_+^2)} ^2\nonumber\\
	 &&+\int_{\mathbb R_{x_1}} \comii{x_1}^\ell\overline{\partial_1^{m+1}u^{I,0}_1}(t,x_1) \inner{\partial_y\varphi_m}(t, x_1, 0) \,dx_1.
\end{eqnarray} 
Now we check the boundary value of $\partial_y \varphi$.  In view of 
\eqref{defver} we see
\begin{eqnarray*}
	\partial_y\varphi_m\big |_{y=0}=\comii{x_1}^\ell \partial_1^m\partial_y^2 u \big|_{y=0}.
\end{eqnarray*}
And moreover, using the relation
\begin{eqnarray*}
	\comii{x_1}^\ell\partial_y^2 u\big |_{y=0}=\partial_t u(x_1, 0)-\overline{\partial_3 u^{I,0}_3} u(x_1, 0)+\overline{u^{I,0}_1}(\partial_1u)(x_1, 0)
\end{eqnarray*}
which follows from \eqref{linsys},  we conclude 
\begin{eqnarray*}
	\partial_y\varphi_m\big |_{y=0}=\partial_t \comii{x_1}^\ell\partial_1^mu(t, x_1, 0)-\comii{x_1}^\ell\partial_1^m \inner{\overline{\partial_3 u^{I,0}_3} u(t,x_1, 0)}+\comii{x_1}^\ell\partial_1^m\inner{\overline{u^{I,0}_1}(\partial_1u)(t,x_1, 0)}.
\end{eqnarray*}
As a result, 
\begin{eqnarray*}
	&&\int_{\mathbb R_{x_1}} \comii{x_1}^\ell\overline{\partial_1^{m+1}u^{I,0}_1}(t,x_1)\comii{x_1}^\ell \inner{\partial_y\varphi_m}(t, x_1, 0) \,dx_1\\
	&=&\frac{d}{dt}\int_{\mathbb R_{x_1}} \comii{x_1}^\ell\overline{\partial_1^{m+1}u^{I,0}_1}(t,x_1) \comii{x_1}^\ell\partial_1^mu(t,x_1, 0)\,dx_1\\
	&&-\int_{\mathbb R_{x_1}} \comii{x_1}^\ell\overline{\partial_t\partial_1^{m+1}u^{I,0}_1}(t,x_1)\comii{x_1}^\ell \partial_1^mu(t,x_1, 0)\,dx_1\\
	&&-\int_{\mathbb R_{x_1}} \comii{x_1}^\ell\overline{\partial_1^{m+1}u^{I,0}_1}(t,x_1) \partial_1^m \inner{\overline{\partial_3 u^{I,0}_3} u(x_1, 0)}\,dx_1\\
	&&+\int_{\mathbb R_{x_1}} \comii{x_1}^\ell\overline{\partial_1^{m+1}u^{I,0}_1}(t,x_1) \partial_1^m\inner{\overline{u^{I,0}_1}(\partial_1u)(t,x_1, 0)}\,dx_1.
\end{eqnarray*}
Moreover, In view of \eqref{dt},  we can repeat the arguments in Lemma \ref{R8} and Lemma \ref{Rothers}, to obtain, observing $\rho<\tau/4,$  
\begin{eqnarray*}
	\abs{\int_{\mathbb R_{x_1}} \comii{x_1}^\ell\overline{\partial_t\partial_1^{m+1}u^{I,0}_1}(t,x_1)\comii{x_1}^\ell \partial_1^mu(t,x_1, 0)\,dx_1}\leq  C \inner{ \frac{(m-3)!}{\rho^{m-1}}}^2 \norm{u^{I,0}_1}_{\mathcal A_\tau} \abs{u}_{X_{\rho,a,m} },
\end{eqnarray*}
\begin{eqnarray*}
	\abs{\int_{\mathbb R_{x_1}} \comii{x_1}^\ell\overline{\partial_1^{m+1}u^{I,0}_1}(t,x_1) \partial_1^m \inner{\overline{\partial_3 u^{I,0}_3} u(x_1, 0)}\,dx_1}
	\leq  C \inner{ \frac{(m-3)!}{\rho^{m-1}}}^2 \norm{u^{I,0}_1}_{\mathcal A_\tau} \norm{u^{I,0}_3}_{\mathcal A_\tau} \abs{u}_{X_{\rho,a,m} }
\end{eqnarray*}
and 
\begin{eqnarray*}
	\abs{\int_{\mathbb R_{x_1}} \comii{x_1}^\ell\overline{\partial_1^{m+1}u^{I,0}_1}(t,x_1) \partial_1^m\inner{\overline{u^{I,0}_1}(\partial_1u)(t,x_1, 0)}\,dx_1}
	\leq  C \inner{ \frac{(m-3)!}{\rho^{m-1}}}^2 \norm{u^{I,0}_1}_{\mathcal A_\tau} ^2 \abs{u}_{X_{\rho,a,m} }.
\end{eqnarray*}
Combing these inequalities above, we conclude
\begin{eqnarray*}
	&&\int_{\mathbb R_{x_1}} \comii{x_1}^\ell\overline{\partial_1^{m+1}u^{I,0}_1}(t,x_1)\comii{x_1}^\ell \inner{\partial_y\varphi_m}(t, x_1, 0) \,dx_1\\
	&\geq &\frac{d}{dt}\int_{\mathbb R_{x_1}} \comii{x_1}^\ell\overline{\partial_1^{m+1}u^{I,0}_1}(t,x_1) \comii{x_1}^\ell\partial_1^mu(t,x_1, 0)\,dx_1-C \inner{ \frac{(m-3)!}{\rho^{m-1}}}^2\abs{u}_{X_{\rho,a,m} },
	\end{eqnarray*}
  which,  along with \eqref{seesf} and \eqref{fiest},  
  yields the conclusion, completing the proof. 
   \end{proof}
   
   \begin{lem} \label{lem0222}
   Let $	a(t)=a_0-\inner{2a_0^2+C_0}t$ with $C_0$   the constants given in Corollary  \ref{cor0222}.     Then  for any $N$, 
   	     \begin{eqnarray*}
   	&&	\frac{1}{2}\frac{d}{dt} \sum_{m=3}^N\inner{ \frac{\rho^{m-1}}{(m-3)!}}^2\norm{\varphi_m}_{L^2(\mathbb R_+^2)} ^2+\frac{1}{2}\sum_{m=3}^N \inner{ \frac{\rho^{m-1}}{(m-3)!}}^2\norm{\comii{x_1}^\ell e^{a  y^2}\partial_1^m\partial_y^2  u (t)}_{L^2(\mathbb R_+^2)} ^2 \\
&&\qquad\qquad\qquad\qquad-\rho'(t)\sum_{m=3}^N \inner{(m-1)^{1/2} \rho^{-1/2} \frac{\rho^{m-1}}{(m-3)!}\norm{\varphi_m}_{L^2(\mathbb R_+^2)} }^2\\
&\leq &{1\over 4}     \abs{u}_{Z_{\rho,a}}^2 -\frac{d}{dt}\sum_{m=3}^N \inner{ \frac{\rho^{m-1}}{(m-3)!}}^2 \int_{\mathbb R_{x_1}} \comii{x_1}^\ell\overline{\partial_1^{m+1}u^{I,0}_1}(t,x_1) \comii{x_1}^\ell\partial_1^mu(t,x_1, 0)\,dx_1\\
 &&+C \inner{  \abs{\rho'}\rho^{-2} \abs{u}_{X_{\rho,a}} +\abs{u}_{X_{\rho,a}}^2 +  \abs{u}_{X_{\rho,a} }^4}.
   \end{eqnarray*}
   \end{lem}
   
   \begin{proof}
   Using the equality \eqref{maiequ} and Lemma \ref{lowleft}, we obtain 
   \begin{eqnarray*}
   	&&\sum_{m=3}^N \inner{ \frac{\rho^{m-1}}{(m-3)!}}^2 \com{\frac{1}{2}\frac{d}{dt} \norm{\varphi_m}_{L^2(\mathbb R_+^2)} ^2+  \norm{\partial_y\varphi_m}_{L^2(\mathbb R_+^2)} ^2-a'(t)\norm{y  \varphi_m}_{L^2(\mathbb R_+^2)}^2}\\
   	&\leq&  -\frac{d}{dt}\sum_{m=3}^N \inner{ \frac{\rho^{m-1}}{(m-3)!}}^2 \int_{\mathbb R_{x_1}} \comii{x_1}^\ell\overline{\partial_1^{m+1}u^{I,0}_1}(t,x_1) \comii{x_1}^\ell\partial_1^mu(t,x_1, 0)\,dx_1\\
 &&+\sum_{m=3}^N(2m-2)\frac{\rho'(t) \rho^{2m-3}}{\left[(m-3)!\right]^2}  \int_{\mathbb R_{x_1}} \comii{x_1}^\ell\overline{\partial_1^{m+1}u^{I,0}_1}(t,x_1) \comii{x_1}^\ell\partial_1^mu(t,x_1, 0)\,dx_1+\abs{u}_{\rho,a}\\
 &&+\sum_{m=3}^N \inner{ \frac{\rho^{m-1}}{(m-3)!}}^2  \inner{\sum_{k=1}^{10}\inner{   \mathcal R_k^m(t), ~\varphi_m (t)}_{L^2(\mathbb R_+^2)}}\\
   	&\leq&  -\frac{d}{dt}\sum_{m=3}^N \inner{ \frac{\rho^{m-1}}{(m-3)!}}^2 \int_{\mathbb R_{x_1}} \comii{x_1}^\ell\overline{\partial_1^{m+1}u^{I,0}_1}(t,x_1) \comii{x_1}^\ell\partial_1^mu(t,x_1, 0)\,dx_1\\
 &&+\sum_{m=3}^N(2m-2)\frac{\rho'(t) \rho^{2m-3}}{\left[(m-3)!\right]^2}  \int_{\mathbb R_{x_1}} \comii{x_1}^\ell\overline{\partial_1^{m+1}u^{I,0}_1}(t,x_1) \comii{x_1}^\ell\partial_1^mu(t,x_1, 0)\,dx_1+\abs{u}_{\rho,a}\\
 &&+ {1\over 4}     \abs{u}_{Z_{\rho,a}}^2 +C_0 \sum_{m=3}^N \inner{ \frac{\rho^{m-1}}{(m-3)!}}^2 \norm{y\varphi_m}_{L^2(\mathbb R_+^2)}^2+C  \abs{u}_{Z_{\rho,a}} \abs{u}_{Y_{\rho,a} }^2+C \inner{\abs{u}_{X_{\rho,a}}^2 +  \abs{u}_{X_{\rho,a} }^4},
   \end{eqnarray*}
   the last inequality following from Corollary \ref{cor0222}. 
   On the other hand,  
   \begin{eqnarray*}
   	&&\sum_{m=3}^N(2m-2)\frac{\rho'(t) \rho^{2m-3}}{\left[(m-3)!\right]^2}  \int_{\mathbb R_{x_1}} \comii{x_1}^\ell\overline{\partial_1^{m+1}u^{I,0}_1}(t,x_1) \comii{x_1}^\ell\partial_1^mu(t,x_1, 0)\,dx_1 \\
   	&\leq& C \sum_{m=3}^N \frac{2^m \rho'(t) \rho^{2m-3}}{\left[(m-3)!\right]^2} \norm{\comii{x_1}^\ell \overline{\partial_1^{m+1}u^{I,0}_1}}_{L^2\inner{\mathbb R_{x_1}}} \inner{\norm{\comii{x_1}^\ell \partial_1^mu}_{L^2}+\norm{\comii{x_1}^\ell \partial_1^m\partial_y u}_{L^2}}\\
   	&\leq& C \sum_{m=3}^N  \frac{2^m \rho'(t) \rho^{m-2}}{\tau^{m+3}}  \norm {u^{I,0}_1}_{G_\tau} \left[\frac{    \rho^{m-1}}{(m-3)!}  \inner{\norm{\comii{x_1}^\ell \partial_1^mu}_{L^2}+\norm{\comii{x_1}^\ell \partial_1^m\partial_y u}_{L^2}}\right]
   	\\
   	&\leq& C  \abs{\rho'}\rho^{-2} \abs{u}_{X_{\rho,a}} 
   	   \end{eqnarray*}   
   	  the last inequality using the fact that $\rho<\tau/3$.   As a result, combining the equalities above  yields
   	   \begin{eqnarray*}
   	&&\sum_{m=3}^N \inner{ \frac{\rho^{m-1}}{(m-3)!}}^2 \com{\frac{1}{2}\frac{d}{dt} \norm{\varphi_m}_{L^2(\mathbb R_+^2)} ^2+  \norm{\partial_y\varphi_m}_{L^2(\mathbb R_+^2)} ^2-a'(t)\norm{y  \varphi_m}_{L^2(\mathbb R_+^2)}^2} \\
   	&\leq& {1\over 4}     \abs{u}_{Z_{\rho,a}}^2  -\frac{d}{dt}\sum_{m=3}^N \inner{ \frac{\rho^{m-1}}{(m-3)!}}^2 \int_{\mathbb R_{x_1}} \comii{x_1}^\ell\overline{\partial_1^{m+1}u^{I,0}_1}(t,x_1) \comii{x_1}^\ell\partial_1^mu(t,x_1, 0)\,dx_1\\
 &&+C_0 \sum_{m=3}^N \inner{ \frac{\rho^{m-1}}{(m-3)!}}^2 \norm{y\varphi_m}_{L^2(\mathbb R_+^2)}^2+C  \abs{u}_{Z_{\rho,a}} \abs{u}_{Y_{\rho,a} }^2+C \inner{  \abs{\rho'}\rho^{-2} \abs{u}_{X_{\rho,a}} +\abs{u}_{X_{\rho,a}}^2 +  \abs{u}_{X_{\rho,a} }^4}.
   \end{eqnarray*}
Moreover from  the relations
\begin{eqnarray*}
	  \norm{\partial_y\varphi_m}_{L^2(\mathbb R_+^2)} ^2 \geq \frac{1}{2} \norm{\comii{x_1}^\ell e^{a  y^2}\partial_1^m\partial_y^2  u (t)}_{L^2(\mathbb R_+^2)} ^2-2a_0^2  \norm{y\varphi_m}_{L^2(\mathbb R_+^2)} ^2
\end{eqnarray*}
and
   \begin{eqnarray*}
   	&&	\frac{1}{2}\frac{d}{dt} \inner{ \frac{\rho^{m-1}}{(m-3)!}}^2\norm{\varphi_m}_{L^2(\mathbb R_+^2)} ^2+ \inner{ \frac{\rho^{m-1}}{(m-3)!}}^2\norm{\partial_y\varphi_m}_{L^2(\mathbb R_+^2)} ^2-a'(t) \inner{ \frac{\rho^{m-1}}{(m-3)!}}^2 \norm{y  \varphi_m}_{L^2(\mathbb R_+^2)}^2\\
&&\qquad\qquad\qquad\qquad-\rho'(t)\inner{(m-1)^{1/2} \rho^{-1/2} \frac{\rho^{m-1}}{(m-3)!}\norm{\varphi_m}_{L^2(\mathbb R_+^2)} }^2\\
&=& \inner{ \frac{\rho^{m-1}}{(m-3)!}}^2 \com{\frac{1}{2}\frac{d}{dt} \norm{\varphi_m}_{L^2(\mathbb R_+^2)} ^2+  \norm{\partial_y\varphi_m}_{L^2(\mathbb R_+^2)} ^2-a'(t)\norm{y  \varphi_m}_{L^2(\mathbb R_+^2)}^2},
   \end{eqnarray*}
  it follows that
   \begin{eqnarray*}
   	&&	\frac{1}{2}\frac{d}{dt} \sum_{m=3}^N\inner{ \frac{\rho^{m-1}}{(m-3)!}}^2\norm{\varphi_m}_{L^2(\mathbb R_+^2)} ^2+\frac{1}{2}\sum_{m=3}^N \inner{ \frac{\rho^{m-1}}{(m-3)!}}^2\norm{\comii{x_1}^\ell e^{a  y^2}\partial_1^m\partial_y^2  u (t)}_{L^2(\mathbb R_+^2)} ^2 \\
   	&&\qquad\qquad\qquad -\inner{a'(t)+2a_0^2}\sum_{m=3}^N \inner{ \frac{\rho^{m-1}}{(m-3)!}}^2\norm{y\varphi_m}_{L^2(\mathbb R_+^2)} ^2 \\
&&\qquad\qquad\qquad\qquad-\rho'(t)\sum_{m=3}^N \inner{(m-1)^{1/2} \rho^{-1/2} \frac{\rho^{m-1}}{(m-3)!}\norm{\varphi_m}_{L^2(\mathbb R_+^2)} }^2\\
&\leq &{1\over 4}     \abs{u}_{Z_{\rho,a}}^2 -\frac{d}{dt}\sum_{m=3}^N \inner{ \frac{\rho^{m-1}}{(m-3)!}}^2 \int_{\mathbb R_{x_1}} \comii{x_1}^\ell\overline{\partial_1^{m+1}u^{I,0}_1}(t,x_1) \comii{x_1}^\ell\partial_1^mu(t,x_1, 0)\,dx_1\\
 &&+C_0 \sum_{m=3}^N \inner{ \frac{\rho^{m-1}}{(m-3)!}}^2 \norm{y\varphi_m}_{L^2(\mathbb R_+^2)}^2+C  \abs{u}_{Z_{\rho,a}} \abs{u}_{Y_{\rho,a} }^2+C \inner{  \abs{\rho'}\rho^{-2} \abs{u}_{X_{\rho,a}} +\abs{u}_{X_{\rho,a}}^2 +  \abs{u}_{X_{\rho,a} }^4}.
   \end{eqnarray*}
   Now observing 
 $
 	a(t)=a_0-\inner{2a_0^2+C_0}t,
$
we complete the proof. 
   \end{proof}
   
   \begin{proof}[{\bf Completion of the proof of Proposition \ref{prenes}}]
   By Lemma	\ref{lem0222}, we integrate both sides over $[0, t]\subset[0,T]$ and then let $N\rightarrow +\infty$, to obtain that  for any $t\in [0,T],$
   \begin{eqnarray*}
   	&&  \sum_{m=3}^{+\infty} \inner{\frac{\rho^{m-1}}{(m-3)!}}^2 \norm{\varphi_m(t)}_{L^2(\mathbb R_+^2)} ^2+\int_0^T \sum_{m=3}^{+\infty} \inner{ \frac{\rho^{m-1}}{(m-3)!}}^2\norm{\comii{x_1}^\ell e^{a  y^2}\partial_1^m\partial_y^2  u (t)}_{L^2(\mathbb R_+^2)} ^2dt\\
 		&&\qquad-\int_0^T \rho'(t) \sum_{m=3}^{+\infty} \frac{m-1}{\rho}\inner{\frac{\rho^{m-1}}{(m-3)!}}^2 \norm{ \varphi_m(t)}_{L^2(\mathbb R_+^2)} ^2 dt\\
 		&\leq &\abs{u_0}_{X_{\rho_0,a_0} }^2+\frac{1}{2}\int_0^T\abs{u(t)}_{Z_{\rho,a}}^2dt+C \int_0^T \inner{\abs{\rho'(t)}\rho^{-2}\abs{u(t)}_{X_{\rho,a} }+\abs{u(t)}_{X_{\rho,a} }^2+ \abs{u(t)}_{X_{\rho,a}}^4}dt\\
 		&&+C \int_0^T  \abs{u(t)}_{Z_{\rho,a}} \abs{u(t)}_{Y_{\rho,a} }^2dt.
   \end{eqnarray*}
   Direct computation also gives 
    \begin{eqnarray*}
   	&&  \sum_{m\leq 2}  \norm{\varphi_m(t)}_{L^2(\mathbb R_+^2)} ^2+\int_0^T \sum_{m\leq 2}  \norm{\comii{x_1}^\ell e^{a  y^2}\partial_1^m\partial_y^2  u (t)}_{L^2(\mathbb R_+^2)} ^2dt-\int_0^T \rho'(t) \sum_{m\leq 2}^{+\infty}   \norm{ \varphi_m(t)}_{L^2(\mathbb R_+^2)} ^2 dt\\
 		&\leq &\abs{u_0}_{X_{\rho_0,a_0} }^2+\frac{1}{2}\int_0^T\abs{u(t)}_{Z_{\rho,a}}^2dt+C \int_0^T \inner{\abs{\rho'(t)}\rho^{-2}\abs{u(t)}_{X_{\rho,a} }+\abs{u(t)}_{X_{\rho,a} }^2+ \abs{u(t)}_{X_{\rho,a}}^4}dt\\
 		&&+C \int_0^T  \abs{u(t)}_{Z_{\rho,a}} \abs{u(t)}_{Y_{\rho,a} }^2dt.
   \end{eqnarray*}
   Similarly,  using the notation
 \begin{eqnarray*}
 	\psi_m =\comii{x_1}^\ell e^{a  y^2}\partial_1^m  u (t),
 \end{eqnarray*}
 we can deduce, following  the proof of  the above two inequalities with slight modification and simpler arguments,  
  \begin{eqnarray*} 
 		&&\sum_{m\leq 2}\norm{\psi_m(t)}_{L^2}^2+\sum_{m=3}^{+\infty} \inner{\frac{\rho^{m-1}}{(m-3)!}}^2 \norm{\psi_m(t)}_{L^2(\mathbb R_+^2)} ^2\\
 		&&\qquad+\int_0^T\inner{ \sum_{m\leq 2}\norm{\comii{x_1}^\ell e^{a  y^2}\partial_1^m\partial_y  u (t)}_{L^2(\mathbb R_+^2)} ^2+\sum_{m=3}^{+\infty} \inner{ \frac{\rho^{m-1}}{(m-3)!}}^2\norm{\comii{x_1}^\ell e^{a  y^2}\partial_1^m\partial_y  u (t)}_{L^2(\mathbb R_+^2)} ^2} dt\\
 		&&\qquad-\int_0^T \rho'(t) \inner{ \sum_{m\leq 2}\norm{ \psi_m}_{L^2}^2+\sum_{m=3}^{+\infty} \frac{m-1}{\rho}\inner{\frac{\rho^{m-1}}{(m-3)!}}^2 \norm{ \psi_m}_{L^2(\mathbb R_+^2)} ^2} dt\\
 		&\leq &\abs{u_0}_{X_{\rho_0,a_0} }^2+\frac{1}{2}\int_0^T\abs{u(t)}_{Z_{\rho,a}}^2dt+C \int_0^T \inner{\abs{\rho'(t)}\rho^{-2}\abs{u(t)}_{X_{\rho,a} }+\abs{u(t)}_{X_{\rho,a} }^2+ \abs{u(t)}_{X_{\rho,a}}^4}dt\\
 		&&+C \int_0^T  \abs{u(t)}_{Z_{\rho,a}} \abs{u(t)}_{Y_{\rho,a} }^2dt
 	  \end{eqnarray*}
Combining these inequalities we conclude, observing  the definition of $\abs{\cdot}_{X_{\rho,a }}, \abs{\cdot}_{Y_{\rho,a }}$ and  $\abs{\cdot}_{Z_{\rho,a }}$ and any $\rho\leq \min\set{\rho_0, \tau/3}$,
\begin{eqnarray*} 
 		&&\abs{u(t)}_{X_{\rho,a }}^2
 		 +\int_0^T \abs{u(t)}_{Z_{\rho,a }}^2dt -\int_0^T \rho'(t) \abs{u(t)}_{Y_{\rho,a }}^2 dt\\
 		&\leq &\abs{u_0}_{X_{\rho_0,a_0} }^2+\frac{1}{2}\int_0^T\abs{u(t)}_{Z_{\rho,a}}^2dt+C \int_0^T \inner{\abs{\rho'(t)}\rho^{-2}\abs{u(t)}_{X_{\rho,a} }+\abs{u(t)}_{X_{\rho,a} }^2+ \abs{u(t)}_{X_{\rho,a}}^4}dt\\
 		&&+C \int_0^T  \abs{u(t)}_{Z_{\rho,a}} \abs{u(t)}_{Y_{\rho,a} }^2dt
 	  \end{eqnarray*}
 	  Thus Claim \eqref{53} follows and the proof is complete.
   \end{proof}

\section{Existence of solution for second component} \label{section5}

In this section, we determine the second component $\Ucal^{p,0}_2$ by solving the parabolic-type equation
$$
	\left\{
	\begin{aligned}
		&\partial_t \Ucal^{p,0}_2 - \partial_y^2 \Ucal^{p,0}_2 + \Ucal^{p,0}_1 \partial_1 \Ucal^{p,0}_2 + \Ucal^{p,1}_3 \partial_y \Ucal^{p,0}_2 = 0,\\		&\Ucal^{p,0}_2(t,x_1,0) = 0,\quad\lim_{y\to +\infty} \Ucal^{p,0}_2(t,x_1,y) = \overline{u^{I,0}_2}(x_1),\\
		&\Ucal^{p,0}_2(0,x_1,y) = u^{B,0}_{0,2} (x_1,y) + \overline{u^{I,0}_{0,2}}(x_1)\, .
	\end{aligned}
	\right.
$$
We recall that 
\begin{equation}
	\label{eq:traceUI02bis} \partial_t\overline{u^{I,0}_2} + \overline{u^{I,0}_1} \partial_1 \overline{u^{I,0}_2} = 0\, .
\end{equation}
Then, the system \eqref{eq:P2} becomes
\begin{equation}
	\label{eq:P2B}\tag{P2bis}
	\left\{
	\begin{aligned}
		&\partial_t u^{B,0}_2 - \partial_y^2 u^{B,0}_2 + \pare{u^{B,0}_1 + \overline{u^{I,0}_1}} \partial_1 u^{B,0}_2 + \pare{u^{B,1}_3 + \overline{u^{I,1}_3} + y \overline{\partial_3u^{I,0}_3}} \partial_y u^{B,0}_2 + \overline{\partial_1 u^{I,0}_2} u^{B,0}_1 = 0,\\
		&\partial_2 u^{B,0}_2 = 0,\\
		&u^{B,0}_2(t,x_1,0) = -\overline{u^{I,0}_2},\quad 
		 \lim_{y\to +\infty} u^{B,0}_2(t,x_1,y) = 0,\\
		&u^{B,0}_2(0,x_1,y) = u^{B,0}_{0,2} (x_1,y).
	\end{aligned}
	\right.
\end{equation}
We have the following results

\begin{thm} \label{th:P2sys}
	Let $\rho_0 > 0$, $a_0 > 0$. For any initial data $u^{B,0}_{2,0} \in X_{\rho_0, a_0}$, there exists $T>0$, $\tau>0$ and $0<a<a_0$, such that the system \eqref{eq:P2B} admits a unique solution $u^{B,0}_2 \in L^{\infty}\pare{[0,T], X_{\rho_0, a}}$.
\end{thm}

\noindent \textbf{Proof.}
In order to prove Theorem \ref{th:P2sys}, the idea is to define an auxiliary function $$v = u^{B,0}_2 + e^{-2a_0y^2} \overline{u^{I,0}_2},$$ which satisfies the following boundary conditions
$$v(t,x_1,0) = \lim_{y\to +\infty} v(t,x_1,y) = 0.$$ Then, the first equation of the system \eqref{eq:P2B} becomes
\begin{multline*}
	\partial_t \pare{v - e^{-2a_0y^2} \overline{u^{I,0}_2}} - \partial_y^2 \pare{v - e^{-2a_0y^2} \overline{u^{I,0}_2}} + \pare{u^{B,0}_1 + \overline{u^{I,0}_1}} \partial_1 \pare{v - e^{-2a_0y^2} \overline{u^{I,0}_2}}\\ 
	+ \pare{u^{B,1}_3 + \overline{u^{I,1}_3} + y \overline{\partial_3u^{I,0}_3}} \partial_y \pare{v - e^{-2a_0y^2} \overline{u^{I,0}_2}} + \overline{\partial_1 u^{I,0}_2} u^{B,0}_1 = 0.
\end{multline*}
Using \eqref{eq:traceUI02bis}, we can rewrite the system \eqref{eq:P2B} as
\begin{equation}
	\label{eq:P2Bv}\tag{$P2_v$}
	\left\{
	\begin{aligned}
		&\partial_t v - \dd_y^2 v + \pare{u^{B,0}_1 + \overline{u^{I,0}_1}} \partial_1 v + \pare{u^{B,1}_3 + \overline{u^{I,1}_3} + y \overline{\partial_3u^{I,0}_3}} \partial_y v + R = 0,\\
		&\partial_2 v = 0,\\
		&v(t,x_1,0) = 0,\quad\lim_{y\to +\infty} u^{B,0}_2(t,x_1,y) = 0,\\
		&v(0,x_1,y) = u^{B,0}_{0,2} (x_1,y) + e^{-2a_0y^2} \overline{u^{I,0}_{0,2}}(x_1),
	\end{aligned}
	\right.
\end{equation}
where
$$ 
R = \pare{16 a_0^2 y^2 - 4a_0} e^{-2a_0y^2} \overline{u^{I,0}_2} + 4a_0 \pare{u^{B,1}_3 + \overline{u^{I,1}_3} + y \overline{\partial_3u^{I,0}_3}} y e^{-2a_0y^2} \overline{u^{I,0}_2} + \pare{1 - e^{-2a_0y^2}} \overline{\partial_1 u^{I,0}_2} u^{B,0}_1.
$$
We remark that the system \eqref{eq:P2Bv} is in the same form as the system \eqref{linsys} with Dirichlet boundary conditions. Thus, we can prove Theorem \ref{th:P2sys} in the same way (with a lot of simplifications) as we did to prove Theorem \ref{th41}. 

 \bigskip
 \noindent {\bf Acknowledgements.}
 The research of the first author was supported by NSF of China(11422106) and Fok Ying Tung Education Foundation (151001), and he  would like to thank the invitation of   `` laboratoire de math\'ematiques Rapha\"el Salem'' of the Universit\'e de Rouen.
 The second author would like to express his sincere thanks to School of mathematics and statistics of Wuhan University for the invitations.
The research of the last author is supported partially by ``The Fundamental Research Funds for Central Universities of China".

\end{document}